\documentclass[12pt,reqno]{amsart}

\usepackage{subfigure}
\usepackage{mathdots} 
\usepackage{tikz}
\usetikzlibrary{matrix,arrows,calc}
\usetikzlibrary{positioning}
\usepackage{makecell}
\usepackage{csquotes}
\usepackage{mathabx}
\usepackage{paralist}
\usepackage{afterpage}

\newcommand\restr[2]{{
  \left.\kern-\nulldelimiterspace 
  #1 
  \vphantom{|} 
  \right|_{#2} 
  }}

\usepackage{chngcntr}
\counterwithin{figure}{section}
\usepackage{float} 

\usepackage{soul}
\usepackage[normalem]{ulem}

\usepackage{amssymb}
\usepackage{graphicx}
\usepackage{geometry} 
\usepackage[all]{xy}

\usepackage{hyperref}  

\usepackage[english]{babel}
\usepackage{amsthm}
\usepackage{amssymb}
\usepackage{graphicx}
\usepackage{subfigure}
\usepackage{geometry} 
\usepackage[all]{xy}
\usepackage{mathdots} 
\usepackage{tikz}
\usetikzlibrary{shapes,snakes}
\usetikzlibrary{matrix,arrows,calc}
\usetikzlibrary{positioning}
\usepackage{makecell}

\usepackage{csquotes}
\usepackage{mathabx}
\usepackage{enumitem}
\usepackage{paralist}
\usepackage{datetime}

\newtheorem{prop}{Proposition}[section]

\newtheorem{thm}[prop]{Theorem}
\newtheorem{cor}[prop]{Corollary}
\newtheorem{lem}[prop]{Lemma}

\newcommand{\NN}{\operatorname{\mathbb{N}}\nolimits}
\newcommand{\ZZ}{\operatorname{\mathbb{Z}}\nolimits}

\newcommand{\RR}{\operatorname{\mathbb{R}}\nolimits}

\newcommand{\RN}[1]{\textup{\uppercase\expandafter{\romannumeral#1}}}

\theoremstyle{definition}
\newtheorem{rem}[prop]{Remark}

\newtheorem{defn}[prop]{Definition}
\newtheorem{ex}[prop]{Example}

\usepackage{chngcntr}
\counterwithin{figure}{section}

\usepackage{float}

\definecolor{myblue}{cmyk}{1.00,0.56,0.00,0.34}
\definecolor{mygreen}{cmyk}{0.5,0,0.5,0.5}
\definecolor{myred}{cmyk}{0.00,1.00,0.63,0.00}
\definecolor{myyellow}{cmyk}{0.00,0.15,1.00,0.00}

\title{Growth behaviour of periodic tame friezes}

\author[K. Baur, K. Fellner, M. J. Parsons, M. Tschabold]{Karin Baur, Klemens Fellner, Mark J. Parsons, Manuela Tschabold}
\address{Institut f\"{u}r Mathematik und Wissenschaftliches Rechnen, 
Universit\"{a}t Graz, NAWI Graz, Heinrichstrasse 36, 
A-8010 Graz, Austria}
\email{baurk@uni-graz.at}
\email{klemens.fellner@uni-graz.at}
\email{markjamesparsons@googlemail.com}
\email{manuela.tschabold@uni-graz.at}


\newcommand{\monthword}[1]{\ifcase#1\or January\or February\or March\or April\or May\or June\or July\or August\or September\or October\or November\or December\fi}
\date{\monthword{\the\month} \the\day, \the\year } 

\begin{document}
\maketitle

\begin{abstract}

We examine the growth behaviour of the entries occurring in $n$-periodic tame friezes of real numbers. 
Extending \cite{T}, we prove that generalised recursive relations exist between all entries of such friezes. These 
recursions are parametrised by a sequence of so-called growth coefficients, which is itself shown to satisfy a recursive relation. Thus, all growth coefficients are determined by a \emph{principal growth coefficient}, 
which can be read-off directly from the frieze.

We place special emphasis on periodic tame friezes of positive integers, specifying the values the growth coefficients take for any such frieze. We establish that the growth coefficients of the pair of friezes arising from a triangulation of an annulus coincide. 
The entries of both are shown to grow asymptotically exponentially, while triangulations of a punctured disc are seen to provide the only friezes of linear growth.
\end{abstract}

\section{Introduction}

Frieze patterns of numbers were first introduced by Coxeter in the early 1970's (\cite{coxeter}). In recent years, 
various generalisations in different contexts have been studied, for example in~\cite{br10, bess, cuntz, hj, MGOST}, to mention just
a few. An excellent overview can be found in the expository article~\cite{MG}, by Morier-Genoud.

\smallskip 

A \emph{frieze} is an array $\mathcal{F} = (m_{i,j})_{i,j}$ of real numbers
\begin{center}
\begin{tikzpicture}[font=\normalsize] 

\matrix(m) [matrix of math nodes,row sep={1.75em,between origins},column sep={1.75em,between origins},nodes in empty cells]{
&0&&0&&0&&0&&0&&&&&\\[-0.25em]
&&1&&1&&1&&1&&1&&&&&\\[-0.25em]
&\node{\cdots}; &&m_{-1,-1}&&m_{0,0}&&m_{1,1}&&m_{2,2}&&m_{3,3}&&\node{\cdots};\\
&&&&m_{-1,0}&&m_{0,1}&&m_{1,2}&&m_{2,3}&& m_{3,4} &&&\\
&&&&&m_{-1,1}&&m_{0,2}&&m_{1,3}&&m_{2,4}&&m_{3,5} \\
&&&&&&&&&&\node[rotate=-6.5,shift={(-0.034cm,-0.08cm)}]  {\ddots};&&&&&\\
};
     
\end{tikzpicture}
\end{center}
\noindent
such that the unimodular rule holds. That is, for every diamond
\begin{center}
\begin{tikzpicture}[font=\normalsize] 

\matrix(m) [matrix of math nodes,row sep={1.75em,between origins},column sep={1.75em,between origins},nodes in empty cells]{
&m_{i+1,j}&\\
m_{i,j}&&m_{i+1,j+1}\\
&m_{i,j+1}&\\
     };
     
\end{tikzpicture}
\end{center}
\noindent
we have $m_{i,j}m_{i+1,j+1} - m_{i,j+1}m_{i+1,j} = 1$. The first non-trivial row of a frieze $\mathcal{F}$, which we denote $(a_i)_{i \in \mathbb{Z}}$ with $a_i = m_{i,i}$ for all $i$, is referred to as the \emph{quiddity row} of the frieze. Such a frieze is called \emph{periodic} if it is invariant under a horizontal translation. In this case, we refer to any finite subsequence of consecutive terms of its quiddity row which generates this row as a \emph{quiddity sequence} for $\mathcal{F}$.

The frieze $\mathcal{F}$ is said to be \emph{finite} if it consists of finitely many rows, ending with a row of 1's followed by a row of 0's. Otherwise, $\mathcal{F}$ is said to be \emph{infinite}. 
If all the entries of $\mathcal{F}$ are positive integers (apart from the initial row of zeros and the final row of zeros in the case of a finite frieze), then it is said to be {\em a frieze of positive integers}.
Finite friezes of positive integers first appeared in the article \cite{coxeter} 
of Coxeter, who then studied them further, together with Conway, in \cite{coco1,coco2}.  They observed that such a finite frieze is invariant under a glide reflection, and hence periodic.

In general, the positive integer condition on entries is too restrictive for our purposes,
and can be replaced by tameness: 
We refer to $\mathcal{F}$ as \emph{tame} (using the terminology introduced in \cite{br10} for $SL_k$-tilings) if any $3\times 3$-matrix formed by successive diagonals has determinant $0$, i.e.
$$
\det\begin{pmatrix}
m_{i,j}&m_{i+1,j}&m_{i+2,j}\\
m_{i,j+1}&m_{i+1,j+1}&m_{i+2,j+1}\\
m_{i,j+2}&m_{i+1,j+2}&m_{i+2,j+2}
\end{pmatrix}=0.$$
Any frieze whose non-trivial entries are all non-zero is tame. In particular, this includes all friezes of positive integers. A tame frieze is periodic if and only if its quiddity row (which fully determines the frieze) is periodic. 

Throughout this article, \emph{our friezes are always assumed to be tame and periodic} (these properties are often tacitly assumed). 
\medskip

Recall that, by \cite{coco1,coco2}, finite friezes of positive integers correspond to triangulations of (convex) polygons: For such a frieze, the quiddity sequence is given by the numbers of triangles incident with each vertex of the corresponding triangulation (with vertices taken in anti-clockwise order around the polygon). In \cite{BPT}, building upon \cite{T}, we gave an analogous characterisation of periodic infinite friezes of positive integers. Namely, that every triangulation of a once-punctured disc with marked points on the outer boundary or every triangulation of an annulus with marked points on both boundaries gives rise to a periodic infinite frieze of positive integers (via a quiddity sequence). Moreover, every such frieze arises from these two types of triangulations: Given such a triangulation, the associated quiddity sequence consists of the numbers of regions incident (in a small neighbourhood) with each vertex on a given boundary component.
\medskip

In this article, we examine the growth behaviour of the entries occurring in a periodic tame frieze. We place special emphasis on finite and infinite friezes of positive integers.

The study of the growth behaviour of periodic friezes was initiated by the fourth author, in \cite{T}. She showed that each diagonal of a periodic infinite frieze, with minimal period $n$ and arising from a triangulation of a punctured disc, is made up of a collection of $n$ arithmetic sequences, with successive terms separated by $n$ positions. Since an entry of an arithmetic sequence is given by twice the previous entry minus the entry preceding that, we immediately deduce a recursive relationship (with coefficient 2) between the entries in a diagonal, see Section~\ref{sec:growthcoefficients} for details.
\medskip

One of our main results is to show that generalised recursive relationships exist for arbitrary periodic tame friezes. For this, it is natural to first extend each such frieze to a ``unimodular lattice'' in the plane. 
Given an infinite frieze $\mathcal{F}=(m_{i,j})_{j-i \geq -2}$, we reflect the entries across the row of zeros while also negating them: For $j-i<-2$ we set $m_{i,j}=-m_{j+2,i-2}$. For finite friezes, we iteratively perform this anti-symmetric mirroring across the rows of zeros in a similar manner, see Figure~\ref{fig:finite-mirror}; in 
\cite{br10}, this extension has been referred to as the 
skew-periodic extension of the finite frieze. 
In both cases, we denote the resulting lattice in the plane by $\mathcal{L} = (m_{i,j})_{i,j\in\ZZ}$. In particular, this allows our results to unanimously apply to finite and infinite friezes.

It is easy to check that the lattice $\mathcal{L}$ of a (tame) frieze $\mathcal{F}$ is also tame in the same sense as for friezes. 
\medskip

\begin{figure}[t]
\scalebox{.9}{\begin{tikzpicture}[font=\normalsize] 
  \matrix(m) [matrix of math nodes,row sep={1.25em,between origins},column sep={1.5em,between origins},nodes in empty cells]{
&&\phantom{-}&&\node[shift={(0cm,0.05cm)}]  {\phantom{-}\vdots};&&\phantom{-}&&\phantom{-}&&\node[shift={(0cm,0.05cm)}]  {\phantom{-}\vdots};&&\phantom{-}&&\phantom{-}\\
& -1&&-1&&-1&&-1&&-1&&-1&&-1&\\
&&\phantom{-}0&&\phantom{-}0&&\phantom{-}0&&\phantom{-}0&&\phantom{-}0&&\phantom{-}0&&\phantom{-}0\\
&\phantom{-}1&&\phantom{-}1&&\phantom{-}1&&\phantom{-}1&&\phantom{-}1&&\phantom{-}1&&\phantom{-}1&\\
\cdots&&\phantom{-}1&&\phantom{-}3&&\phantom{-}1&&\phantom{-}2&&\phantom{-}2&&\phantom{-}1&&\phantom{-}3&\cdots \\
&\phantom{-}1&&\phantom{-}2&&\phantom{-}2&&\phantom{-}1&&\phantom{-}3&&\phantom{-}1&&\phantom{-}2&\\
&&\phantom{-}1&&\phantom{-}1&&\phantom{-}1&&\phantom{-}1&&\phantom{-}1&&\phantom{-}1&&\phantom{-}1\\
&\phantom{-}0&&\phantom{-}0&&\phantom{-}0&&\phantom{-}0&&\phantom{-}0&&\phantom{-}0&&\phantom{-}0&\\
&&-1&&-1&&-1&&-1&&-1&&-1&&-1\\
&-1&&-2&&-2&&-1&&-3&&-1&& -2&\\
\cdots &&-1&&-3&&-1&&-2&&-2&&-1&&-3 &\cdots \\
& -1&&-1&&-1&&-1&&-1&&-1&&-1&\\
&&\phantom{-}0&&\phantom{-}0&&\phantom{-}0&&\phantom{-}0&&\phantom{-}0&&\phantom{-}0&&\phantom{-}0\\
&\phantom{-}1&&\phantom{-}1&&\phantom{-}1&&\phantom{-}1&&\phantom{-}1&&\phantom{-}1&&\phantom{-}1&\\
&&\phantom{-}&&\node[shift={(0cm,-0.05cm)}]  {\phantom{-}\vdots};&&\phantom{-}&&\phantom{-}&&\node[shift={(0cm,-0.05cm)}]  {\phantom{-}\vdots};&&\phantom{-}&&\phantom{-}\\
};

\end{tikzpicture}}
\caption{A lattice of a finite frieze.}
\label{fig:finite-mirror}
\end{figure}
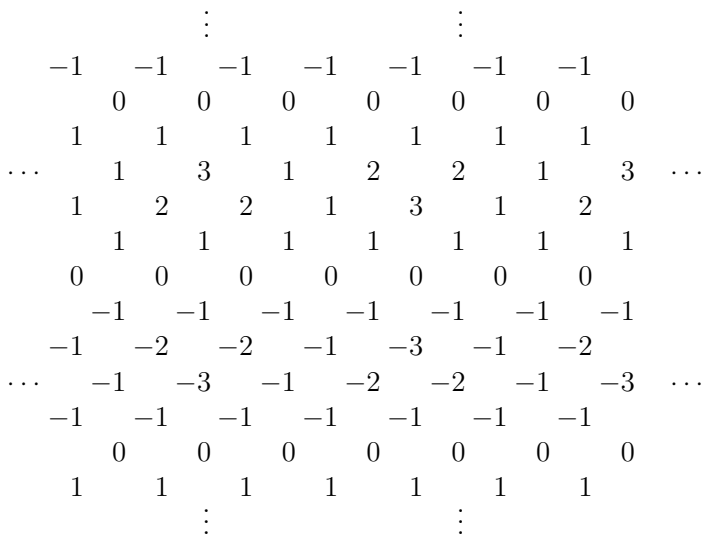

As mentioned above and shown in \cite{T}, for any infinite frieze arising from a triangulation of a punctured disc, there 
holds a recursive relationship (with the constant coefficient 2) throughout each diagonal of the lattice, see Section~\ref{sec:growthcoefficients}.
We shall show that such recursive relationships exist for all periodic tame friezes, yet with possibly non-constant coefficients. In particular, we 
refer to the first coefficient appearing in the recursion as the \emph{principal growth coefficient} of the frieze (and its lattice). 

Surprisingly, the principal growth coefficient of a periodic frieze with minimal period $n$ can be read-off directly from the frieze (and its lattice), as the (constant) difference between entries lying directly above/below one-another in the non-trivial rows 
$n$ and $n-2$ (Theorems~\ref{thm:friezerelationc} and \ref{thm:relation}), with the quiddity row taken as the first row.

For the lattice of an infinite frieze arising only from triangulations of annuli, the principal growth coefficient  is greater than 2 (Corollary~\ref{cor:lem:growthcoefficient}). The principal growth coefficient of the lattice of a finite frieze of positive integers must belong to $\{ -2,0,1 \}$ (Proposition~\ref{prop:finite-case-s}). 

By considering a periodic frieze with minimal period $n$ as $kn$-periodic for $k \in \mathbb{N}$, we obtain a family/sequence of growth coefficients for the frieze (and its lattice), all of which can again be read-off directly from the lattice. 
Note that the principal growth coefficient acts both as the parameter of the corresponding growth coefficient recursion as well 
as the first (non-trivial) entry of the recursively defined sequence of growth coefficients, see Proposition~\ref{prop:relationvaluesA}. This double role of the 
principal growth coefficient has strong implications for the possible behaviour of the sequence of growth coefficients, see Sections~\ref{sec:analysis-periodic} and \ref{sec:analysis-recursion}.

The sequence of growth coefficients  turns out to be uniformly equal to 2 for infinite friezes arising from triangulations of a punctured disc (Proposition~\ref{prop:puncturedDisk}). For an infinite frieze arising only from triangulations of annuli, the family of growth coefficients forms a strictly increasing monotone sequence (Lemma~\ref{lem:properiescoefficents} together with Corollary~\ref{cor:annulus}).  
On the other hand, the family of growth coefficients of a finite frieze of positive integers forms a periodic sequence (Corollary~\ref{cor:lm:finite-case-s}). For any periodic tame frieze, we also define a zeroth growth coefficient, which we take to be 2 ($=1-(-1)$) and can also be read-off from the associated lattice. 
The study of the qualitative and quantitative behaviour of these sequences of growth coefficients is best carried out by employing analytic methods on the corresponding recursive relationship, see Section~\ref{sec:analysis-recursion}.
\medskip

A triangulation of an annulus $A_{n,m}$ ($n,m>0$) in fact gives rise to two periodic infinite friezes: one from the outer boundary and one from the inner boundary. It is natural to expect that there is a strong relationship between these friezes. We show that they can be considered to have the same growth coefficients, when treating the former as $n$-periodic and the latter as $m$-periodic (Theorem~\ref{thm:innerouterfriezecoefficient}). An important idea used in the proof of this result is that the growth coefficient of a periodic infinite frieze can be considered to be invariant under ``cutting" and ``gluing" operations (Theorem~\ref{thm:stablecoefficient}), which respectively correspond to removing or adding a peripheral triangle (that is, removing a triangle formed by three successive points on a boundary, or gluing a triangle to a boundary segment on a boundary)
in any associated triangulation.
 
In Section~\ref{sec:analysis-periodic}, as well as establishing the aforementioned results on the family of growth coefficients of a periodic frieze of positive integers, we also study the growth behaviour of such a frieze in terms of the recursively related growth coefficients. While an infinite frieze arising from a triangulation of a punctured disc can of course be considered to have linear growth, one arising from an annulus grows asymptotically exponentially as a result of the larger growth coefficients (Theorem~\ref{thm:growth} and Proposition~\ref{rem:growth}).
 
The proofs of the results of Section~\ref{sec:analysis-periodic} are partly based on the results of Section~\ref{sec:analysis-recursion}, which contains a more general analysis of the recursion defining the sequences of growth coefficients. 
While Proposition~\ref{prop:S} classifies conditions under which generalised recursion sequences grow 
linearly or asymptotically exponentially, Proposition~\ref{prop:S2} 
studies periodic solutions. The proof of the latter is built on 
Theorem~\ref{thm:two-real-sequences}, which provides useful generalised multi-stage recursion formulas.  
Theorem~\ref{thm:two-real-sequences} is a interesting result on its own, since it seems to serve as a starting point for a characterisation of all possible finite friezes.
\medskip

Caldero and Chapoton \cite{CaCh} established a connection between finite friezes of positive integers and cluster algebras of type $A$. Baur and Marsh \cite{BM} subsequently extended this idea by producing modified (branched) finite frieze patterns associated to cluster algebras of type $D$. Following the initial appearance of the present article, Gunawan, Musiker and Vogel \cite{GMV} have initiated the study of the connection between periodic infinite friezes and cluster algebras of type $D$ (punctured disc case) and affine type $A$ (annulus case). By considering infinite friezes whose entries are cluster algebra elements, they obtain geometric and cluster algebraic interpretations of a number of our main ideas and results. In particular, the growth coefficients of an infinite frieze are shown to correspond to so-called ``bracelets" in the associated triangulated surface, which in turn correspond to important cluster algebra elements. The geometric interpretation of growth coefficients as bracelets also gives a geometric justification of our result (Theorem~\ref{thm:innerouterfriezecoefficient}) that the sequences of growth coefficients coincide for the pair of friezes associated to a triangulation of an annulus. The recurrence relations we obtain in Theorem~\ref{thm:relation} are interpretted as part of a broader class of relations between cluster algebra elements.

\section{Growth coefficients of periodic friezes}\label{sec:growthcoefficients}

We consider a lattice $\mathcal{L}=(m_{i,j})_{i,j}$ in the plane obtained from the (infinite or finite) 
frieze $\mathcal F=(m_{i,j})_{i,j}$ with quiddity row $(a_i)_i$, where $a_i=m_{i,i}$ for all $i \in \ZZ$. 
Recall that $\mathcal F$, and 
hence also $\mathcal L$, are always assumed to be periodic and tame. 
Let us point out that for infinite friezes, $\mathcal F$ is the sublattice of $\mathcal L$ formed by the $m_{i,j}$ with $j-i\ge -2$. 
For finite friezes, the sublattice of $\mathcal L$ formed by the $m_{i,j}$ with $j-i\ge -2$ consists of infinitely many copies of 
the finite frieze, mirrored ``downwards'', with signs reversed as appropriate. 
When working with lattices obtained from finite friezes, we will often tacitly work in the 
sublattice of $\mathcal L$ formed by the $m_{i,j}$ with $j-i\ge -2$.

We have the following three frieze properties in $\mathcal L$, for all $i\le j$ and every $i\leq k\leq j+1$,

{\centering
\begin{gather}\label{friezerelationa}
m_{i,j}=
{\small 
  \det 
  \begin{pmatrix}a_i&1&&&0\\
  1&a_{i+1}&1&\\
  &\ddots&\ddots&\ddots\\
  &&1&a_{j-1}&1\\
  0&&&1&a_{j}
  \end{pmatrix}
},\\[0.5em]\label{friezerelationb}
m_{i,j}=a_j m_{i,j-1}-m_{i,j-2}, \quad \text{and} \quad m_{i,j}=a_i m_{i+1,j}-m_{i+2,j},\\[0.5em]\label{friezerelationc} 
m_{i,j}=m_{i,k-1}m_{k,j}-m_{i,k-2}m_{k+1,j}.
\end{gather}}

\noindent In the finite case, (\ref{friezerelationa}) and (\ref{friezerelationb}) already appeared in \cite{coxeter}. For tame lattices, (\ref{friezerelationb}) is given by Equation (5) of~\cite{br10}, with (\ref{friezerelationa}) following as a trivial consequence. Using (\ref{friezerelationb}), the proof of (\ref{friezerelationc}) given in \cite{BPT} (for finite and infinite friezes of positive integers) extends readily to our setting. We note that (\ref{friezerelationb}) is a linear difference equation which is sometimes called a discrete Hill equation (\cite{MGOST}).

Our aim is to consider the growth behaviour of the entries within a diagonal of a (periodic tame) frieze and the associated lattice. 
In diagonals of friezes arising from triangulations of punctured discs, see \cite{T}, 
arithmetic sequences of non-negative integers appear, providing linear growth in our setting.
In particular, due to \cite[Proposition 3.11]{T}, given such a frieze $\mathcal{F}=(m_{i,j})_{i,j}$ of period $n$, then for any entry $m_{i,j}$ there exists a positive integer $d$, depending on $i$ and $j$, such that $m_{i,j+kn}=m_{i,j}+kd$ for all $k\in\mathbb{Z}_{\ge 0}$.
In other words, if we fix the $i$-th South-East diagonal in 
$\mathcal{F}$ and choose some arbitrary $j\in\{i-2,i-1,\dots , i+n-3\}$, we get an arithmetic sequence $(m_{i,j+kn})_{k\ge0}$.
This actually gives the following recursion formula $m_{i,j+(k+2)n}=2m_{i,j+(k+1)n}-m_{i,j+kn}$. 
Moreover, this  recursion still holds if we replace $n$ with any multiple of $n$. 

This is not true in general for periodic friezes. For instance, for 
periodic infinite friezes of positive integers not arising from triangulations of punctured discs, we will 
see that a similar recursion holds with a coefficient strictly greater than 2. As motivation, we consider the special case of 
one-periodic friezes.

\begin{ex}\label{ex:1periodicfriezes}
Let $\mathcal{F}=(m_{i,j})_{i,j}$ be a one-periodic 
frieze with quiddity sequence $(a)$, where $a\in \RR$ and $\mathcal{L}=(m_{i,j})_{i,j}$ its associated lattice. 
By (\ref{friezerelationb}), we have the following recursion formula for a sequence of three consecutive entries in an SE-diagonal of $\mathcal{L}$:
$$m_{i,j+2}=am_{i,j+1}-m_{i,j}$$
for all $i,j\in\ZZ$. In case $a$ is an integer strictly greater than $2$, we note that the given quiddity sequence arises from a triangulation of an annulus, but not of a punctured disc (refer to \cite[Corollary 4.5]{BPT}).
\end{ex}

The next theorem tells us that the difference between the non-trivial rows $n$ and $(n-2)$ in an $n$-periodic frieze (and hence also its lattice) is constant.

\begin{thm}\label{thm:friezerelationc}
Let $\mathcal{L}=(m_{i,j})_{i,j}$ be the lattice of an $n$-periodic tame frieze $\mathcal F$.  
Then
$$m_{1,n}-m_{2,n-1}=m_{k+1,k+n}-m_{k+2,k+n-1}$$
for all $k\in\ZZ$.
\end{thm}

\begin{proof}
Let $(a_1,a_2,\dots, a_n)$ be the quiddity sequence of $\mathcal{F}$. By (\ref{friezerelationb}), we have 
$m_{1n}=a_1m_{2,n}-m_{3,n}$ and $m_{2,n+1}=a_1m_{2,n}-m_{2,n-1},$ where in the latter relation we used that $a_1=a_{n+1}$. It follows immediately that $m_{1,n}-m_{2,n-1}=a_1m_{2,n}-m_{3,n}-(a_1m_{2,n}-m_{2,n+1})=m_{2,n+1}-m_{3,n}.$ All other equalities may be established similarly. 
\end{proof}

If we choose $n$ to be minimal in Theorem~\ref{thm:friezerelationc}, we get a family of invariants for the frieze. 
Hence it makes sense to abbreviate them by introducing the following notion of growth coefficients.

\begin{defn}\label{defn:growthcoefficient}
Let $\mathcal{F}=(m_{i,j})_{i,j}$ be a periodic tame frieze with minimal period $n$ and let $\mathcal L$ be its lattice.  
For $k\ge 0$, the \emph{$k$th growth coefficient} for $\mathcal{F}$ (and for $\mathcal L$) 
is given by $s_0:=2$ and $s_k:=m_{1,kn}-m_{2,kn-1}$, otherwise. In particular, 
we say that $s_1$ is the \emph{principal growth coefficient} for $\mathcal{F}$ (and for $\mathcal L$).
\end{defn}

\begin{rem}
For a one-periodic infinite frieze, the principal growth coefficient $s_1$ coincides with the single entry $a$ in the quiddity sequence, see Example~\ref{ex:1periodicfriezes}. 
A more general example illustrating the growth coefficients is given in Figure~\ref{fig:friezegrowthcoefficients}. 
An immediate consequence of (\ref{friezerelationa}) is that each growth coefficient is a difference of two determinants.
\end{rem}

For the upcoming results to hold, it is not required for $n$ to be the minimal period of the frieze. Therefore, we will often use a more relaxed notion by assigning growth coefficients to quiddity sequences: 
Given a frieze $\mathcal F=(m_{i,j})_{i,j}$ with lattice $\mathcal L$ and 
quiddity sequence $q$, 
we define $s_q:=m_{1,n}-m_{2,n-1}$, where $n$ is the length of $q$. In particular, if $n$ is minimal (that is, if $n$ is the minimal period of $\mathcal{F}$), we have $s_q=s_1$.

One of our main results  is the following theorem giving a linear recursion formula for the entries in a diagonal of the lattice associated to a 
frieze depending on a growth coefficient. 
It motivates our use of the terminology ``growth coefficient''.

\begin{thm}\label{thm:relation}
Let $\mathcal{L}=(m_{i,j})_{i,j}$ be the lattice of an $n$-periodic tame frieze $\mathcal{F}$ with
quiddity sequence $q$ of length $n$. Then for all $i,j$, we have

\begin{inparaenum}[$(i)$]
\item $m_{i,j+2n}=s_qm_{i,j+n}-m_{i,j},$

\item $m_{i-2n,j}=s_qm_{i-n,j}-m_{i,j}.$
\end{inparaenum}
\end{thm}

\begin{proof}
We initially consider $(i)$. Without loss of generality, we will establish the claim for the diagonal with $i=1$. This entails showing that $m_{1,j+2n}=(m_{1,n}-m_{2,n-1})m_{1,j+n}-m_{1,j}$ for all $j$.

We first deal with the case $j\geq -1$, using induction on $j$. We require two initial cases.

Let $j=-1$. Since $m_{1,-1}=0$, we aim to show $m_{1,2n-1}=(m_{1,n}-m_{2,n-1})m_{1,n-1}$. Taking $i=1,j=2n-1,k=n$ in (\ref{friezerelationc}) and using periodicity, we obtain $m_{1,2n-1}=m_{1,n-1}m_{n,2n-1}-m_{1,n-2}m_{n+1,2n-1}=(m_{n,2n-1}-m_{1,n-2})m_{1,n-1}$. By Theorem~$\ref{thm:friezerelationc}$, the latter is equal to $(m_{1,n}-m_{2,n-1})m_{1,n-1}$ as required.

Suppose $j=0$. Since $m_{1,0}=1$, it is enough to show that  $m_{1,2n}=(m_{1,n}-m_{2,n-1})m_{1,n}-1$. As above we make use of (\ref{friezerelationc}), taking $i=1,j=2n,k=n+1$, to obtain $m_{1,2n}=m_{1,n}m_{n+1,2n}-m_{1,n-1}m_{n+2,2n}=m_{1,n}^2-m_{1,n-1}m_{2,n},$ where the latter follows by periodicity. But by the unimodular rule, we have $m_{1,n-1}m_{2,n}=m_{1,n}m_{2,n-1}+1$, so $m_{1,2n}=m_{1,n}^2-m_{1,n}m_{2,n-1}-1=(m_{1,n}-m_{2,n-1})m_{1,n}-1$ as desired.

Our induction hypothesis is that $m_{1,j+2n-1}=(m_{1,n}-m_{2,n-1})m_{1,j+n-1}-m_{1,j-1}$ and $m_{1,j+2n}=(m_{1,n}-m_{2,n-1})m_{1,j+n}-m_{1,j}$. We will establish that $m_{1,j+2n+1}=(m_{1,n}-m_{2,n-1})m_{1,j+n+1}-m_{1,j+1}.$ By (\ref{friezerelationb}) and periodicity, we have $m_{1,j+n+1}=a_{j+1}m_{1,j+n}-m_{1,j+n-1}$ and $m_{1,j+1}=a_{j+1}m_{1,j}-m_{1,j-1}$. Using these two relations together with the induction hypothesis, we obtain $(m_{1,n}-m_{2,n-1})m_{1,j+n+1}-m_{1,j+1}=(m_{1,n}-m_{2,n-1}) (a_{j+1}m_{1,j+n}-m_{1,j+n-1})-(a_{j+1}m_{1,j}-m_{1,j-1})=a_{j+1}((m_{1,n}-m_{2,n-1})m_{1,n+j}-m_{1,j})-((m_{1,n}-m_{2,n-1}) m_{1,n+j-1}-m_{1,j-1})=a_{j+1}m_{1,2n+j}-m_{1,2n+j-1}=m_{1,2n+j+1}$ as required, where the latter follows again by (\ref{friezerelationb}) and periodicity.

It follows that $(i)$ holds in the sublattice $\mathcal{L}_+$ for which $j-i\geq -2$. (That is, the sublattice bounded above by the central row of zeros $(m_{i,i-2})_{i\in\mathbb{Z}}$.) Likewise, $(ii)$ also holds in $\mathcal{L}_+$. Using a ``mirroring" argument, it follows easily that $(i)$ also holds in the sublattice $\mathcal{L}_-$ where $j-i\leq -2$: Consider the diagonal with $i=1$, for $j\leq -1-2n$ (so that $j+2n \leq -1$). We have $s_q m_{1,j+n}-m_{1,j}=-(s_q m_{j+n+2,-1}-m_{j+2,-1})=-m_{j+2n+2,-1}=m_{1,j+2n}$. Likewise, $(ii)$ also holds in $\mathcal{L}_-$.

In order to establish $(i)$ for the whole lattice, it remains to deal with the cases for which $m_{1,j}$, $m_{1,j+n}$ and $m_{1,j+2n}$ don't all lie in the same sublattice $\mathcal{L}_+$ or $\mathcal{L}_-$. Suppose first that $m_{1,j}$ lies in $\mathcal{L}_- \setminus \mathcal{L}_+$, while $m_{1,j+n}$ and $m_{1,j+2n}$ lie in $\mathcal{L}_+$, so that $-1-n\leq j < -1$ (that is, $n-1\leq j+2n < 2n-1$). In this case, we use strong induction on $j$ to establish
\begin{equation}
\label{recclaim}
m_{1,j+2n} = s_q m_{1,j+n} - m_{1,j}.
\end{equation}

Our first initial case, $j=-1-n$, is checked by noting that $m_{1,n-1} = s_q \cdot 0 + m_{1,n-1} = s_q m_{1,-1} + m_{1-n,-1} = s_q m_{1,-1} + m_{1,-1-n}$. And our second initial case, $j=-n$, is checked by noting that $m_{1,n} = (m_{1,n}-m_{2,n-1}) \cdot 1 + m_{2,n-1} = (m_{1,n}-m_{2,n-1}) \cdot 1 + m_{-n+2,-1} = s_q m_{1,0} - m_{1,-n}$.

Now, suppose that $1-n \leq j < -1$ and that (\ref{recclaim}) holds for $j-1$ and $j-2$. Using this together with (\ref{friezerelationb}), we have
\begin{eqnarray*}
m_{1,j+2n} & = & a_{j+2n} m_{1,j+2n-1} - m_{1,j+2n-2}\\
& = & a_{j+2n} (s_q m_{1,j+n-1}-m_{1,j-1})-(s_q m_{1,j+n-2}-m_{1,j-2})\\
& = & s_q (a_{j+2n} m_{1,j+n-1}-m_{1,j+n-2})-(a_{j+2n} m_{1,j-1}-m_{1,j-2})\\
& = & s_q (a_{j+2n} m_{1,j+n-1}-m_{1,j+n-2})-(a_{j} m_{1,j-1}-m_{1,j-2})\\
& = & s_q m_{1,j+n} - m_{1,j}.
\end{eqnarray*}

It remains to consider the case where $m_{1,j}$ and $m_{1,j+n}$ lie in $\mathcal{L}_-$ and $m_{1,j+2n}$ lies in $\mathcal{L}_+ \setminus \mathcal{L}_-$. That is, the case where $-1-2n < j \leq -1-n$ (so that $-1 < j+2n \leq n-1$). For such $j$, using the SW-diagonal analogue of the above argument, we have $m_{j+n+2,n-1} = s_q m_{j+2n+2,n-1}-m_{j+3n+2,n-1}$. Therefore, since $m_{1,j+2n} = -m_{j+2n+2,-1} = -m_{j+3n+2,n-1}$, we see that $m_{1,j+2n} = s_q (-m_{j+2n+2,n-1})-(-m_{j+n+2,n-1}) = s_q m_{n+1,j+2n} - m_{n+1,j+n} = s_q m_{1,j+n} - m_{1,j}$ as required.

The proof of $(ii)$ can be completed using analogous arguments.
\end{proof}

We state Theorem \ref{thm:relation} for the special case where the period is minimal as a corollary.  
\begin{cor}\label{cor:s-k-relation}
Given the lattice $\mathcal{L} = (m_{i,j})_{i,j}$ of a periodic tame frieze $\mathcal{F}$ with minimal period $n$, we have
$$m_{i,j+2kn}=s_km_{i,j+kn}-m_{i,j}, \text{ and } m_{i-2kn,j}=s_km_{i-kn,j}-m_{i,j}$$
for all $k \geq 0$.
\end{cor}

\begin{rem}\label{rem:growthcoefficients}
It is immediate that all of the growth coefficients of an infinite frieze arising from a triangulation of a punctured disc are equal to $2$. 
\end{rem}

We illustrate Theorem \ref{thm:relation} and Corollary \ref{cor:s-k-relation} with an example.

\begin{ex}
\label{ex:growthex}
Let us consider the frieze 
in Figure \ref{fig:friezegrowthcoefficients}. The minimal period of the frieze is 3, and its first four growth coefficients 
are given by $s_0=2$, $s_1=3$, $s_2=7$ and $s_3=18$. In particular, the quiddity sequence $q = (1,2,6,1,2,6)$, 
or indeed any quiddity sequence of length 6, gives rise to the growth coefficient 
$s_q=s_2=7$. The sequences of entries highlighted in green and red are recursively related by $s_1$ and $s_q=s_2$ respectively.
\end{ex}

\begin{rem}
\label{rem:linexprecrel}
Let $\mathcal{L} = (m_{i,j})_{i,j}$ be the lattice of a given $n$-periodic tame frieze $\mathcal{F}$ with quiddity sequence $q$ of length $n$. Since the entries of a diagonal separated by $n$ positions are recursively related as in Theorem~\ref{thm:relation}, then so are linear expressions in entries separated by $n$ positions. That is, for $i_k, j_k \in \mathbb{Z}$ and $\lambda_k\in\RR$, for $N \in \mathbb{N}$ and $1\leq k \leq N$, we have
$$\sum_{k =1}^ N{\lambda_k m_{i_k,j_k+2n}} = s_q \left (\sum_{k =1}^ N{\lambda_k m_{i_k,j_k+n}} \right ) - \sum_{k =1}^ N{\lambda_k m_{i_k,j_k}}.$$
The analogous statement for the South-West diagonals also holds.

It follows that the (arithmetic) row-means arising from a fundamental domain (of width $n$, or equivalently, any positive integer multiple of $n$) for $\mathcal{F}$ are also recursively related via $s_q$.
This further motivates our approach to understanding the growth behaviour of periodic tame friezes.
\end{rem}

The next result, which will play an important role in further analysing the growth behaviour of the entries of a frieze (and its lattice), shows how the growth coefficients of a frieze are recursively related to each other. This recursive relationship again involves the principal growth coefficient $s_1$ as a coefficient. 
Moreover, we use this relationship to provide a 
closed formula for the $k$th growth coefficient $s_k$ in terms of the principal growth coefficient.

\begin{prop}\label{prop:relationvaluesA}
Let $\mathcal{F}$ be a periodic tame frieze.  
Then we have \smallskip

\begin{inparaenum}[$(a)$]
\item\label{prop:relationvaluesA1} $s_{k+2}=s_1s_{k+1}-s_k,$ \quad for $k\ge 0$,
{$s_0=2$,}

\item\label{prop:relationvaluesA2} $\displaystyle{s_k = s_1^k + k\sum_{l=1}^{\lfloor{\frac{k}{2}}\rfloor} (-1)^{l} \frac{1}{k-l}{k-l \choose l}s_1^{k-2l}},$ \quad for $k\ge 1$.
\end{inparaenum}
\end{prop}

\begin{proof}
The first statement is immediate in view of Remark~\ref{rem:linexprecrel} (with $q$ of course taken to be of minimal length, so that $s_q = s_1$). The proof of the second statement is technical but straightforward; it is included in Appendix~\ref{App:A}.
\end{proof}

\begin{figure}[t]
\scalebox{.9}{\begin{tikzpicture}[font=\normalsize] 
  \matrix(m) [matrix of math nodes,row sep={1.5em,between origins},column sep={1.5em,between origins},nodes in empty cells]{
 -1&&-1&&-1&&-1&&-1&&-1&&&&&&&&&&&\\
&0&&0&&0&&0&&0&&0&&&&&&&&&&&\\
&&1&&1&&1&&1&&1&&1&&&&&&&&&&\\
&&&1&&2&&6&&1&&2&&6&&&&&&&&&\\
&&&&1&&11&&5&&1&&11&&5&&&&&&&&\\
&&&&&5&&9&&4&&5&&9&&4&&&&&&&\\
&&&&&&4&&7&&19&&4&&7&&19&&&&&&\\
&&&&&&&3&&33&&15&&3&&33&&15&&&&&\\
&&&&&&&&14&&26&&11&&14&&26&&11&&&&\\
&&&&&&&&&11&&19&&51&&11&&19&&51&&&\\
&&&&&&&&&&8&&88&&40&&8&&88&&40&&\\
&&&&&&&&&&&37&&69&&29&&37&&69&&29&\\
&&&&&&&&&&&&29&&50&&134&&29&&50&&134\\
&&&&&&&&&&&&&&&\node[rotate=-6.5,shift={(-0.034cm,-0.08cm)}]  {\ddots};&&&&&&\node[rotate=-6.5,shift={(-0.034cm,-0.08cm)}]  {\ddots};\\
};

\draw[opacity=0,rounded corners,fill=myblue,fill opacity=0.15] (m-1-1.south west) -- (m-13-13.south west) -- (m-13-19.south west) -- (m-1-7.south west) -- cycle;

\draw ($(m-3-3)+(0,0.6125cm)$) node[ellipse, minimum height=2cm,minimum width=0.75cm,draw,semithick,opacity=0.5] {};
\draw ($(m-6-6)+(0,0.6125cm)$) node[ellipse, minimum height=2cm,minimum width=0.75cm,draw,semithick,opacity=0.5] {};
\draw ($(m-9-9)+(0,0.6125cm)$) node[ellipse, minimum height=2cm,minimum width=0.75cm,draw,semithick,opacity=0.5] {};
\draw ($(m-12-12)+(0,0.6125cm)$) node[ellipse, minimum height=2cm,minimum width=0.75cm,draw,semithick,opacity=0.5] {};

\draw (m-2-1) node[shift={(-1.5cm,0cm)}]{$s_0=2$};
\draw (m-5-4) node[shift={(-1.5cm,0cm)}]{$s_1=3$};
\draw (m-8-7) node[shift={(-1.5cm,0cm)}]{$s_2=7$};
\draw (m-11-10) node[shift={(-1.5cm,0cm)}]{$s_3=18$};

\fill[myred,opacity=0.3] (m-1-9) circle (0.25cm);
\fill[myred,opacity=0.3] (m-7-15) circle (0.25cm);
\fill[myred,opacity=0.3] (m-13-21) circle (0.25cm);

\fill[mygreen,opacity=0.3] (m-2-8) circle (0.25cm);
\fill[mygreen,opacity=0.3] (m-5-11) circle (0.25cm);
\fill[mygreen,opacity=0.3] (m-8-14) circle (0.25cm);
\fill[mygreen,opacity=0.3] (m-11-17) circle (0.25cm);
\end{tikzpicture}}
\caption{Growth coefficients (and the associated recursions) in the lattice of a periodic infinite frieze.}
\label{fig:friezegrowthcoefficients}
\end{figure}

\begin{ex}
\label{ex:skintermsofs1}
To illustrate Proposition~\ref{prop:relationvaluesA}, 
we compute the first few growth coefficients in terms of $s_1$: 
\begin{align*}
{s_0} & {=2,}\\
{s_1} & {=s_1,}\\
s_2 &= s_1^2-2, \\ 
s_3 &= s_1^3 - 3s_1,\\
s_4 &= s_1^4 - 4s_1^2 + 2,\\
s_5 &= s_1^5 - 5s_1^3 + 5s_1, \\
s_6 &= s_1^6 - 6s_1^4 + 9s_1^2 - 2.
\end{align*}
\end{ex}

\begin{rem}
As a consequence of Proposition~\ref{prop:relationvaluesA}(\ref{prop:relationvaluesA1}), it follows by \cite[Proposition 2.34]{msw} that, for $k \geq 0$, $s_k = 2 T_k(\frac{s_1}{2})$, where $T_k(x)$ denotes the $k$th Chebyshev polynomial of the first kind. 
\end{rem}

\section{Relating the pair of friezes of a triangulated annulus}\label{sec:pairoffriezes}

In this section, we concentrate solely on periodic infinite friezes of positive integers. 
Recall that every such frieze arises from a triangulation of a punctured disc or of an annulus, where we can assume without loss of generality that the set of marked points on both boundaries is non-empty.
In the latter case, two periodic infinite friezes can be associated to the same triangulation: one for the outer boundary and one for the inner boundary. In this section, we show that they can be considered to have the same growth coefficients. 
We observe that if we can obtain the two friezes arising from a triangulation of an annulus via triangulations of a punctured 
disc, then all growth coefficients are equal to 2 for both friezes.
Hence we already know the claim is true in this case. Let us point out this happens if and only if there is no 
{\em bridging arc} in the triangulation of the annulus, i.e.\ there is no arc connecting the two boundaries. 
So from now on, we will assume for this section, 
that the triangulation of the annulus contains bridging arcs (for details we refer to \cite[Sections 3 and 4]{BPT}).

We first establish that we can add or remove triangles from a triangulation of an annulus (or a punctured disc) without changing the 
growth behaviour: 
In \cite[Section 2]{T}, we described two algebraic operations on periodic infinite friezes of positive integers, called $n$-gluing and $n$-cutting. 
In the language of triangulations, these operations respectively add or remove a peripheral triangle in the triangulation 
associated to the periodic infinite frieze we are modifying. 
In particular, given a quiddity sequence $q=(a_1,\dots,a_n)$, $n$-gluing above a pair $(a_i,a_{i+1})$ 
produces the quiddity sequence $\hat q=(a_1+2,1)$, if $n=1$, and 
$\hat q=(a_1,\dots, a_i+1,1,a_{i+1}+1,\dots ,a_n)$, otherwise. 
Whenever we have an entry $a_i=1$ in $q$, for some $i\in\{1,\dots,n\}$, we can perform an $n$-cutting above $a_i$ and the outcome 
is the quiddity sequence $\check q=(a_{i+1}-2)$, if $n=2$, and 
$\check q=(a_1,\dots,a_{i-1}-1,a_{i+1}-1,a_n)$, otherwise. (We recall that a quiddity sequence of an infinite frieze of positive integers cannot contain two consecutive ones.)

The next result states that these operations have no effect on the growth coefficient associated to the quiddity sequences.

\begin{thm}\label{thm:stablecoefficient}
Given a periodic infinite frieze of positive integers with quiddity sequence $q$ of length $n$, let $\hat q$ and $\check q$ be quiddity sequences obtained from 
$q$ by performing $n$-gluing and $n$-cutting (if defined), respectively. Then $s_q=s_{\hat q}$ and $s_q=s_{\check q}$.
\end{thm}

As a direct consequence, it is enough to compute growth coefficients for {\em bridging triangulations}, i.e.\ for triangulations where all 
arcs are bridging. 
In the proof of Theorem \ref{thm:stablecoefficient}, we will use the following lemma whose prove we omit as it is straightforward.

\begin{lem}\label{lm:fractions}
Let $a,b,c,d,e,f\in\mathbb{Z}\setminus\{0\}$. Then, for $c\ne\pm f$, the following are equivalent

\begin{inparaenum}[$(a)$]
\item $\displaystyle \frac{a+b}{c}=\frac{d+e}{f}$,

\item $\displaystyle \frac{a-d+b-e}{c-f}=\frac{a+b}{c}$,

\item $\displaystyle \frac{a+d+b+e}{c+f}=\frac{a+b}{c}$.
\end{inparaenum}
\end{lem}

\begin{proof}[Proof of Theorem \ref{thm:stablecoefficient}]
Let $\mathcal{F}=(m_{i,j})_{i,j}$ be an $n$-periodic infinite frieze of positive integers  with quiddity sequence $q=(a_1,a_2,\dots, a_n)$. Assume there is $k\in\{1,2,\dots,n\}$ such that $a_k=1$. By \cite[Proposition 2.9]{T}, we have the $(n-1)$-periodic frieze $\widecheck{\mathcal{F}}=(\check m_{i,j})_{i,j}$ with quiddity sequence $\check q=(a-2)$, if $n=2$, or $\check q=(a_1,\dots,a_{k-2},a_{k-1}-1,a_{k+1}-1,a_{k+2},\dots a_n)$, otherwise. Clearly, $\mathcal{F}$ is again obtained from $\widecheck{\mathcal{F}}$ by $(n-1)$-gluing above $(a_{k-1}-1,a_{k+1}-1)$, so we can make use of \cite[Corollary 2.5]{T} which relates the entries of the two friezes as follows
$$m_{i,j}=\begin{cases}
   \check m_{i_{k+1},j_{k-1}} &\text{if } i\not \equiv k+1, j \not \equiv k-1,\\
   \check m_{i_{k+1},j_{k-1}-1}+\check m_{i_{k+1},j_{k-1}} &\text{if } i\not \equiv k+1, j \equiv k-1,\\
   \check m_{i_{k+1}-1,j_{k-1}}+\check m_{i_{k+1},j_{k-1}} &\text{if } i \equiv k+1,j \not \equiv k-1,\\
   \check m_{i_{k+1}-1,j_{k-1}-1}+\check m_{i_{k+1},j_{k-1}}+\check m_{i_{k+1}-1,j_{k-1}}\\+\check m_{i_{k+1},j_{k-1}-1} &\text{if } i \equiv k+1, j \equiv k-1,
   \end{cases}$$
where $i_x=i-t$ whenever $x+(t-1)n< i\le x+tn$ and $\equiv$ means equal reduced modulo $n$. If $k\not\in\{1,n\}$, we have
$s_q=m_{1,n}-m_{2,n-1}=\check m_{1_{k+1},n_{k-1}}-\check m_{2_{k+1},(n-1)_{k-1}}=\check m_{1,n-1}-\check m_{2,n-2}=s_{\check q}$. Suppose $k=1$. It follows that $s_q=m_{1,n}-m_{2,n-1}=\check m_{1_2,n_0-1}+\check m_{1_2,n_0}-\check m_{2_2-1,(n-1)_0}-\check m_{2_2,(n-1)_0}=\check m_{1,n-2}+\check m_{1,n-1}-\check m_{1,n-2}-\check m_{2,n-2}=s_{\check q}$. Let $k=n$. We get $s_q=m_{1,n}-m_{2,n-1}=\check m_{1_{n+1}-1,n_{n-1}}+\check m_{1_{n+1},n_{n-1}}-\check m_{2_{n+1},(n-1)_{n-1}-1}-\check m_{2_{n+1},(n-1)_{n-1}}=\check m_{1,n-1}+\check m_{2,n-1}-\check m_{2,n-2}-\check m_{2,n-1}=s_{\check q}$.

That $s_q=s_{\hat q}$ can be shown analogously.
\end{proof}

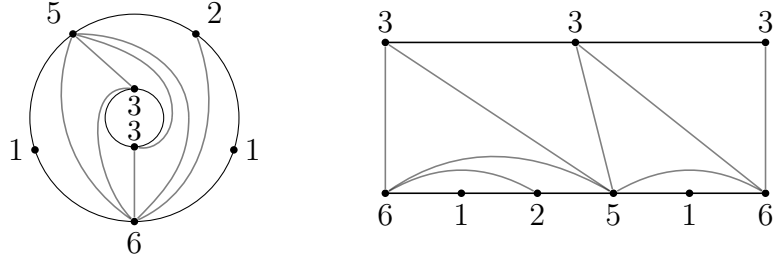
\begin{figure}[t]
\begin{tikzpicture}[font=\normalsize] 

\begin{scope}[scale=1.1,shift={(-3,0)}]

\draw (0,0) circle (1.25cm);
\foreach \x in {-144,-72,0,72,144} {
   \begin{scope}[rotate=\x]
    \node (\x) at (0,-1.25) [fill,circle,inner sep=1pt] {};
   \end{scope}
}

\draw (144) node [above right] {$2$};
\draw (72) node [right] {$1$};
\draw (0) node [below] {$6$};
\draw (-72) node [left] {$1$};
\draw (-144) node [above left] {$5$};

\draw (0,0) circle (0.35cm);
\node (a) at (0,0.35) [fill,circle,inner sep=1pt] {};
\node (b) at (0,-0.35) [fill,circle,inner sep=1pt] {};

\draw[semithick,opacity=0.5] (a) to (-144);
\draw[semithick,opacity=0.5] (0) to (b);
\draw[semithick,opacity=0.5,out=135,in=175] (0) to  (a);

\draw[semithick,opacity=0.5,out=-15,in=85] (-144) to (0.45,0);
\draw[semithick,opacity=0.5,out=-10,in=-95] (b) to (0.45,0);

\draw[semithick,opacity=0.5,out=25,in=-70] (0) to  (144);
\draw[semithick,opacity=0.5,out=145,in=-110] (0) to  (-144);

\draw[semithick,opacity=0.5,out=45,in=-90] (0) to (0.65,0);
\draw[semithick,opacity=0.5,out=0,in=90] (-144) to (0.65,0);

\draw (a) node [below,shift={(0,0.05)}] {\small$3$};
\draw (b) node [above,shift={(0,-0.05)}] {\small $3$};

\end{scope}

\draw[semithick] (0,1) to (5,1);

\node (a) at (0,1) [fill,circle,inner sep=1pt] {};
\node (b) at (2.5,1) [fill,circle,inner sep=1pt] {};
\node (c) at (5,1) [fill,circle,inner sep=1pt] {};

\draw[semithick] (0,-1) to (5,-1);

\foreach \x in {0,1,2,3,4,5} {
   \begin{scope}[shift={(\x cm, 0 cm)}]
    \node (\x) at (0,-1) [fill,circle,inner sep=1pt] {};
   \end{scope}
}

\draw (a) node [above] {$3$};
\draw (b) node [above] {$3$};
\draw (c) node [above] {$3$};
\draw (0) node [below] {$6$};
\draw (1) node [below] {$1$};
\draw (2) node [below] {$2$};
\draw (3) node [below] {$5$};
\draw (4) node [below] {$1$};
\draw (5) node [below] {$6$};

\draw[semithick,opacity=0.5] (0) to (a);
\draw[semithick,opacity=0.5] (5) to (c);
\draw[semithick,opacity=0.5] (3) to (a);
\draw[semithick,opacity=0.5] (3) to (b);
\draw[semithick,opacity=0.5] (5) to (b);

\draw[semithick,opacity=0.5,out=30,in=150] (0) to (2);
\draw[semithick,opacity=0.5,out=35,in=150] (0) to (3);
\draw[semithick,opacity=0.5,out=30,in=150] (3) to (5);

\end{tikzpicture} 
\caption{A triangulation of $A_{5,2}$ providing outer quiddity sequence $q=(6,1,2,5,1)$ and inner quiddity sequence $\bar q=(3,3)$.}\label{fig:pairofsequences}
\end{figure}

Let $T$ be a triangulation of $A_{n,m}$, where we assume $n$ and $m$ to be positive. 
Label the vertices of the outer and inner boundaries 
anti-clockwise by $\{1,2,\dots, n\}$ and $\{1,2,\dots,m\}$ respectively. 
The $n$-tuple $q=(a_1,a_2,\dots, a_n)$ 
with $a_i$ the number of triangles incident with vertex $i$  
on the outer boundary is called the \emph{outer quiddity sequence} of length $n$. 
The 
\emph{inner quiddity sequence} $\bar q=(\bar a_1,\bar a_2,\dots , \bar a_m)$ of length 
$m$ is given by the numbers $\bar a_i$ of triangles incident with the vertices 
on the inner boundary. 
We name the corresponding friezes generated by $q$ and $\bar q$ the 
\emph{outer frieze} and \emph{inner frieze} respectively, and denote them by 
$\mathcal{F}$ and $\bar{\mathcal{F}}$.

Every triangulation of $A_{n,m}$ can be viewed as a periodic 
triangulation of an infinite strip in the plane by using the cylindrical 
representation of the annulus, and translates of it, in the plane. For details, 
see \cite[Section 3.3]{BPT}. This point of view is very convenient for determining 
entries in the associated friezes as they are given by ``matching numbers''. 
We will make frequent use of this. 
We will refer to a single copy of the triangulation in the strip as a fundamental domain 
(for $T$).

Clearly, considering triangulations of $A_{n,m}$, the entries in the outer quiddity sequence and the entries in the inner quiddity sequence are closely related. In particular, if the entries in one are large, that means the other must have a comparatively high number of small entries, in order to balance. Indeed, one can easily convince oneself that the average value of all $a_i$'s and $\bar a_i$'s is equal to three. Another remarkable connection between the pair of outer and inner friezes is illustrated in the example below and later stated in Theorem \ref{thm:innerouterfriezecoefficient}.

\begin{ex}
We consider the triangulation of $A_{5,2}$ given in Figure \ref{fig:pairofsequences}, yielding the two quiddity sequences $q=(6,1,2,5,1)$ and $\bar q=(3,3)$ for the outer and inner boundary respectively. A ``fundamental domain'' for each associated frieze is given in the figure below. We observe that the growth coefficients of the two quiddity sequences coincide. However, note that it is necessary to keep the quiddity sequences exactly as they come from the triangulation, regardless of their period.
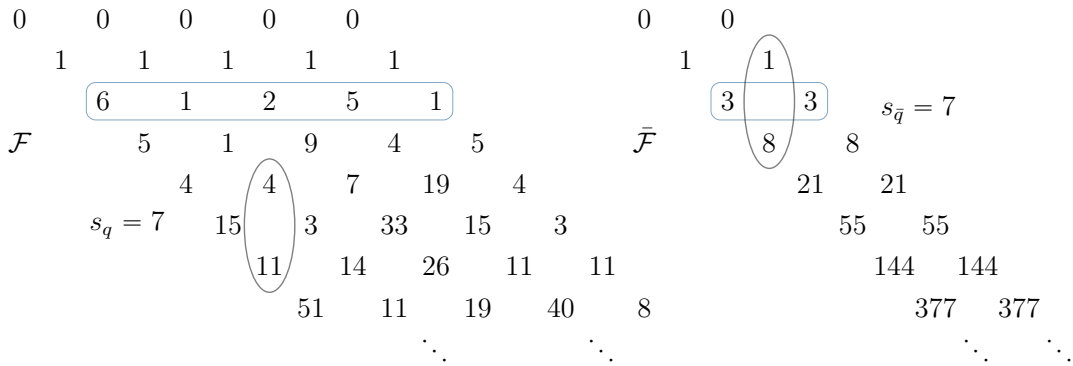
\begin{center}
\resizebox{\textwidth}{!}{\begin{tikzpicture}[font=\normalsize] 
  \matrix(m) [matrix of math nodes,row sep={1.5em,between origins},column sep={1.5em,between origins},nodes in empty cells]{
0&&0&&0&&0&&0&&&&&&&0&&0&&&&&&&&&\\
&1&&1&&1&&1&&1&&&&&&&1&&1&&&&&&&&\\
&&6&&1&&2&&5&&1&&&&&&&3&&3&&&&&&&\\
\node{\mathcal{F}};&&&5&&1&&9&&4&&5&&&&\node{\bar{\mathcal{F}}};&&&8&&8&&&&&\\
&&&&4&&4&&7&&19&&4&&&&&&&21&&21&&&&\\
&&&&&15&&3&&33&&15&&3&&&&&&&55&&55&&&\\
&&&&&&11&&14&&26&&11&&11&&&&&&&144&&144&&\\
&&&&&&&51&&11&&19&&40&&8&&&&&&&377&&377&\\
&&&&&&&&&&\node[rotate=-6.5,shift={(-0.034cm,-0.08cm)}]  {\ddots};&&&&\node[rotate=-6.5,shift={(-0.034cm,-0.08cm)}]  {\ddots};&&&&&&&&&\node[rotate=-6.5,shift={(-0.034cm,-0.08cm)}]  {\ddots};&&\node[rotate=-6.5,shift={(-0.034cm,-0.08cm)}]  {\ddots};\\
};

\draw[opacity=0.5,rounded corners,myblue] ($(m-3-3.north west)+(0cm,0.cm)$) -- ($(m-3-11.north east)+(0cm,0cm)$) --($(m-3-11.south east)+(0cm,0cm)$) --  ($(m-3-3.south west)+(0cm,0cm)$) -- cycle;
\draw[opacity=0.5,rounded corners,myblue] ($(m-3-18.north west)+(0cm,0.cm)$) -- ($(m-3-20.north east)+(0cm,0cm)$) --($(m-3-20.south east)+(0cm,0cm)$) --  ($(m-3-18.south west)+(0cm,0cm)$) -- cycle;

\draw ($(m-7-7)+(0,0.6125cm)$) node[ellipse, minimum height=2cm,minimum width=0.75cm,draw,semithick,opacity=0.5] {};
\draw (m-6-6) node[shift={(-1.5cm,0cm)}]{$s_q=7$};

\draw ($(m-4-19)+(0,0.6125cm)$) node[ellipse, minimum height=2cm,minimum width=0.75cm,draw,semithick,opacity=0.5] {};
\draw (m-3-25) node[shift={(-1.5cm,0cm)}]{$s_{\bar q}=7$};
\end{tikzpicture}}
\end{center}
\end{ex}

\begin{thm}\label{thm:innerouterfriezecoefficient}
Let $\mathcal{F}=(m_{i,j})_{i,j}$ and $\bar{\mathcal{F}}=(\bar m_{i,j})_{i,j}$ be the outer and inner boundary pair of friezes associated to a triangulation of $A_{n,m}$ $(m>0)$ with quiddity sequences $q$ and $\bar q$, respectively. Then $s_q=s_{\bar q}$.
\end{thm}

Before commencing the proof of this result, we  introduce some notions concerning bridging triangulations in order to provide the necessary ingredients. Let $T$ be a triangulation of an annulus $A_{n,m}$ with $n,m>0$. Due to Theorem \ref{thm:stablecoefficient}, we can assume $T$ to be a bridging triangulation. So a fundamental domain for $T$ in the infinite strip
may be viewed as a collection of $r$ \emph{alternating fans} $(r\ge1)$ as illustrated in Figure \ref{figurealternatingfans} with $n=N_r=\sum_{i=1}^rn_i$ and $m=M_r=\sum_{i=1}^rm_i$, where the $n_i$'s and $m_i$'s are the numbers of triangles in the fans (note that $n_i=N_{i}-N_{i-1}$ and $m_i=M_{i}-M_{i-1}$ for each $i$). Without loss of generality, we may assume there is a fan starting at $1$ at each boundary (as shown).

\afterpage{
\begin{figure}[t]
\begin{tikzpicture}[scale=1.25,font=\normalsize] 

\draw[semithick] (-0.5,1) to (10.5,1);

\foreach \x in {0,0.5,1.5,2,2.5,3.5,4,6,6.5,7.5,8,8.5,9.5,10} {
   \begin{scope}[shift={(\x cm, 0 cm)}]
    \node (\x) at (0,1) [fill,circle,inner sep=1pt] {};
   \end{scope}
}

\draw (0,1) node [above,shift={(0 cm, 0.065 cm)}] {$1$};
\draw (0.5,1) node [above,shift={(0 cm, 0.065 cm)}] {$2$};
\draw (1.2,0.5) node [above] {$\cdots$};
\draw (1.5,1) node [above] {$M_1$};
\draw (3.2,0.5) node [above] {$\cdots$};
\draw (3.5,1) node [above] {$M_2$};
\draw (7.5,1) node [above] {$M_{r-1}$};
\draw (9.5,1) node [above] {$M_r$};
\draw (10,1) node [above,shift={(0 cm, 0.065 cm)}] {$1$};
\draw (7.2,0.5) node [above] {$\cdots$};
\draw (9.2,0.5) node [above] {$\cdots$};

\draw[semithick] (-0.5,-1) to (10.5,-1);

\foreach \x in {0,0.5,1.5,2,2.5,3.5,4,6,6.5,7.5,8,8.5,9.5,10} {
   \begin{scope}[shift={(\x cm, 0 cm)}]
    \node (\x) at (0,-1) [fill,circle,inner sep=1pt] {};
   \end{scope}
}

\draw (0,-1) node [below,shift={(0 cm, -0.05 cm)}] {$1$};
\draw (0.5,-1) node [below,shift={(0 cm, -0.05 cm)}] {$2$};
\draw (0.85,-0.5) node [below] {$\cdots$};
\draw (1.5,-1) node [below,shift={(0 cm, -0.05 cm)}] {$N_1$};
\draw (2.85,-0.5) node [below] {$\cdots$};
\draw (3.5,-1) node [below,shift={(0 cm, -0.05 cm)}] {$N_2$};
\draw (6.85,-0.5) node [below] {$\cdots$};
\draw (7.5,-1) node [below,shift={(0 cm, -0.05 cm)}] {$N_{r-1}$};
\draw (8.85,-0.5) node [below] {$\cdots$};
\draw (9.5,-1) node [below,shift={(0 cm, -0.05 cm)}] {$N_r$};
\draw (10,-1) node [below,shift={(0 cm, -0.05 cm)}] {$1$};

\draw (5,0) node{$\dots$};

\foreach \x in {0,0.5,1.5,2}{
\draw[semithick,opacity=0.5] (0 cm,1cm) to (\x cm,-1cm);
\draw[semithick,opacity=0.5] (\x cm,1cm) to (2 cm,-1cm);
\begin{scope}[shift={( 6 cm, 0 cm)}]
\draw[semithick,opacity=0.5] (0 cm,1cm) to (\x cm,-1cm);
\draw[semithick,opacity=0.5] (\x cm,1cm) to (2 cm,-1cm);
\end{scope}   
}

\foreach \x in {2,2.5,3.5,4}{
\draw[semithick,opacity=0.5] (2 cm,1cm) to (\x cm,-1cm);
\draw[semithick,opacity=0.5] (\x cm,1cm) to (4 cm,-1cm);
\begin{scope}[shift={( 6 cm, 0 cm)}]
\draw[semithick,opacity=0.5] (2 cm,1cm) to (\x cm,-1cm);
\draw[semithick,opacity=0.5] (\x cm,1cm) to (4 cm,-1cm);
\end{scope}
}

\draw[semithick] (0 cm,-1cm) to (0 cm,1cm) to (2 cm,-1cm) to (2 cm,1cm) to (4 cm,-1cm) to (4 cm,1cm);
\begin{scope}[shift={( 6 cm, 0 cm)}]
\draw[semithick] (0 cm,-1cm) to (0 cm,1cm) to (2 cm,-1cm) to (2 cm,1cm) to (4 cm,-1cm) to (4 cm,1cm);
\end{scope}

\end{tikzpicture}
\caption{The fundamental domain of a bridging triangulation of $A_{n,m}$.}
\label{figurealternatingfans}
\end{figure}
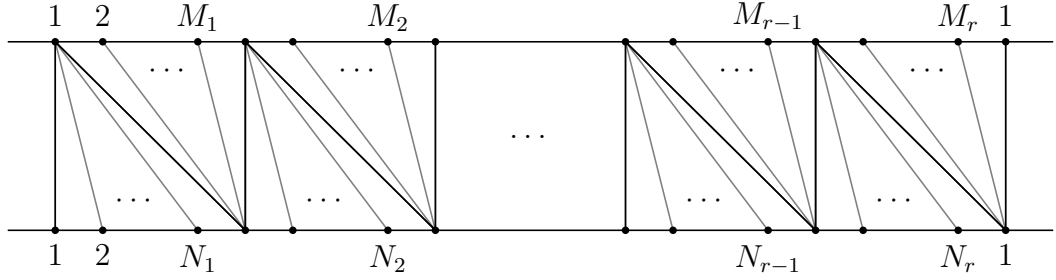
}

Recall that, by \cite[Theorem 5.6]{BPT}, for the outer frieze $\mathcal{F}=(m_{i,j})_{i,j}$ associated to $T$, we have that $m_{i,j}$ is the number of matchings between the corresponding vertices in the infinite strip and triangles incident with them  
for all $i\le j$, with an analogous statement holding for the inner frieze $\bar{\mathcal{F}}=(\bar m_{i,j})_{i,j}$. Let $q$ and $\bar q$ be the quiddity sequences of  $\mathcal{F}$ and $\bar{\mathcal{F}}$ respectively.

Our strategy for the proof will be as follows: We will express $s_q=m_{1,n}-m_{2,n-1}$ in terms of $n_1,\dots,n_r,m_1,\dots,m_r$ by counting matchings, and then show that the expression we obtain is invariant under performing the substitution $n_i \leftrightarrow m_{r+1-i}$ for all $1\le i\le r$. This suffices due to the fact the the expression obtained in this manner must clearly be $s_{\bar q}$. (Imagine rotating Figure~\ref{figurealternatingfans} through $180^\circ$; to be precise, we obtain $\bar m_{2,m+1}-\bar m_{3,m}$ which is equal to  $s_{\bar q}$ by Proposition \ref{thm:friezerelationc}.)

Using the fan structure of the triangulation, one can routinely establish the following results by iteratively building up matching numbers (noting that we take empty sums to be 0).
\begin{align}\label{relfan1}
m_{1,N_k}&=m_{1,N_{k-1}}+n_k\left(\left(\sum_{i=1}^{k-1}m_im_{1,N_i}\right)+m_r+1\right)\\\nonumber
&=1+\sum_{j=1}^k n_j\left(\left(\sum_{i=1}^{j-1}m_im_{1,N_i}\right)+m_r+1\right),\\\label{relfan2}
m_{2,N_k}&=m_{2,N_{k-1}}+n_k\left(1+\sum_{i=1}^{k-1}m_im_{2,N_i}\right)=\sum_{j=1}^k n_j\left(1+\sum_{i=1}^{j-1}m_im_{2,N_i}\right),\\\label{relfan3}
m_{2,N_k-1}&=\sum_{j=1}^{k-1}n_j\left(1+\sum_{i=1}^{j-1}m_im_{2,N_i}\right)+(n_k-1)\left(1+\sum_{i=1}^{k-1}m_im_{2,N_i}\right)\\\nonumber
&=m_{2,N_k}-\left(1+\sum_{i=1}^{k-1}m_im_{2,N_i}\right).
\end{align}

In the following, we will use the notion of continuants to represent the determinants of tridiagonal matrices:  the $k$-th continuant is the multivariate polynomial $P_k(x_1,x_2,\dots,x_k)$ defined recursively by $P_0=1$, $P_1(x_1)=x_1,$ and $P_k(x_1,x_2,\dots,x_k)=x_kP_{k-1}(x_1,x_2,\dots,x_{k-1})+P_{k-2}(x_1,x_2,\dots,x_{k-2})$. In particular, $P_k(x_1,x_2,\dots,x_k)=P_k(x_k,x_{n-1},\dots,x_1)$. Moreover, it is known that
$$P_k(x_1,x_2,\dots,x_k)=\det \begin{pmatrix}x_1&1&&&0\\
-1&x_2&1&&\\
&\ddots&\ddots&\ddots&\\
&&-1&x_{k-1}&1\\
0&&&-1&x_k
\end{pmatrix}.$$

\begin{lem}\label{lem:determinantrelationA}
Let $T$ be a bridging triangulation of $A_{n,m}$ $(m>0)$, as in Figure \ref{figurealternatingfans}, with associated outer frieze $\mathcal{F}=(m_{i,j})_{i,j}$.
For $1\le k\le r$, let $N_k=\sum_{i=1}^kn_i$. Then

\begin{inparaenum}[$(i)$]
\item\label{lem:determinantrelationA1}
$m_{1,N_k}=P_{2k}(m_r+1,n_1,m_1,n_2,m_2,\dots,m_{k-1},n_k),$

\item\label{lem:determinantrelationA2}
$\displaystyle 1+m_r+\sum\limits_{i=1}^k m_im_{1,N_i}=P_{2k+1}(m_r+1,n_1,m_1,\dots,m_{k-1},n_k,m_k)$,

\item\label{lem:determinantrelationA3}
$m_{2,N_k}=P_{2k-1}(n_1,m_1,n_2,m_2,\dots, m_{k-1},n_k),$

\item\label{lem:determinantrelationA4}
$\displaystyle 1+\sum\limits_{i=1}^k m_im_{2,N_i}=P_{2k}(n_1,m_1,n_2,m_2,\dots, m_{k-1},n_k,m_k)$.
\end{inparaenum}
\end{lem}

\begin{proof}
We prove (\ref{lem:determinantrelationA1}) and (\ref{lem:determinantrelationA2}) together using induction on $k$. First, for $k=1$, observe that, $P_{2}(m_r+1,n_1)=n_1P_{1}(m_r+1)+P_{0}=1+n_1(m_r+1)=m_{1,N_1}$ (by (\ref{relfan1})) and also $P_{3}(m_r+1,n_1,m_1)=m_1P_{2}(m_r+1,n_1)+P_{1}(m_r+1)=1+m_r+m_1m_{1,N_1}.$
Suppose (\ref{lem:determinantrelationA1}) and (\ref{lem:determinantrelationA2}) hold for $k-1$. Then, by the induction hypothesis, we get
\begin{align*}
P_{2k}(m_r+1,n_1,m_1,\dots,m_{k-1},n_k)&=n_kP_{2k-1}(m_r+1,n_1,m_1,\dots,n_{k-1},m_{k-1})
\\&\quad+P_{2k-2}(m_r+1,n_1,m_1,\dots,m_{k-2},n_{k-1})\\
&=n_k\left(1+m_r+\sum\limits_{i=1}^{k-1} m_im_{1,N_i}\right)+m_{1,N_{k-1}}\\
&\stackrel{(\ref{relfan1})}{=}m_{1,N_k}.
\end{align*}

In turn, we have

\begin{align*}
P_{2k+1}(m_r+1,n_1,m_1,\dots,n_k,m_k)&=m_kP_{2k}(m_r+1,n_1,m_1,\dots,m_{k-1},n_k)\\
&\quad +P_{2k-1}(m_r+1,n_1,m_1,\dots,n_{k-1},m_{k-1})\\
&=1+m_r+\sum\limits_{i=1}^k m_im_{1,N_i}.
\end{align*}

That (\ref{lem:determinantrelationA3}) and (\ref{lem:determinantrelationA4}) hold can be shown in a similar manner.
\end{proof}

It is now a simple matter to prove Theorem \ref{thm:innerouterfriezecoefficient}.

\begin{proof}[Proof of Theorem \ref{thm:innerouterfriezecoefficient}]
Taking $k=r$ in $(\ref{relfan3})$ gives $s_q=m_{1,n}-m_{2,n-1}=m_{1,N_r}-m_{2,N_r-1}=m_{1,N_r}-\left(m_{2,N_r}-\left(1+\sum_{i=1}^{r-1}m_im_{2,N_i}\right)\right)=m_{1,N_r}-m_{2,N_r}+\left(1+\sum_{i=1}^{r-1}m_im_{2,N_i}\right).$ We use Lemma \ref{lem:determinantrelationA} (\ref{lem:determinantrelationA1}), (\ref{lem:determinantrelationA3}) and (\ref{lem:determinantrelationA4}) to obtain
\begin{align*}
s_q&=P_{2r}(m_r+1,n_1,m_1,n_2,m_2,\dots,m_{r-1},n_r)-P_{2r-1}(n_1,m_1,n_2,m_2,\dots, m_{r-1},n_r)\\
&\quad+P_{2r-2}(n_1,m_1,n_2,m_2,\dots, m_{k-1},n_{r-1},m_{r-1})\\
&=(m_r+1)P_{2r-1}(n_1,m_1,n_2,m_2,\dots,m_{r-1},n_r)+P_{2r-2}(m_1,n_2,m_2,\dots,m_{r-1},n_r)\\
&\quad-P_{2r-1}(n_1,m_1,n_2,m_2,\dots, m_{r-1},n_r)+P_{2r-2}(n_1,m_1,n_2,m_2,\dots, m_{k-1},n_{r-1},m_{r-1})\\
&=m_rP_{2r-1}(n_1,m_1,n_2,m_2,\dots,m_{r-1},n_r)+P_{2r-2}(m_1,n_2,m_2,\dots,m_{r-1},n_r)\\
&\quad+P_{2r-2}(n_1,m_1,n_2,m_2,\dots, m_{k-1},n_{r-1},m_{r-1})\\
&=P_{2r}(n_1,m_1,n_2,m_2,\dots,m_{r-1},n_r,m_r)+P_{2r-2}(m_1,n_2,m_2,\dots,m_{r-1},n_r).
\end{align*}

Recall that by performing the substitutions $n_i\leftrightarrow m_{r+1-i}$ for all $1\le i\le r$, we obtain $s_{\bar q}$. Therefore
$$s_{\bar q}=
P_{2r}(m_r,n_r,m_{r-1},n_{r-1},\dots,n_2,m_1,n_1)+P_{2r-2}(n_r,m_{r-1},n_{r-1},\dots,n_2,m_1).$$
It is immediate that $s_q=s_{\bar q}$.
\end{proof}

From the proof of Theorem~\ref{thm:innerouterfriezecoefficient}, we have the following corollary.

\begin{cor}\label{cor:growthcoefficient}
Let $\mathcal{F}$ and $\bar{\mathcal{F}}$ be the outer and inner boundary pair of friezes associated to a bridging triangulation $T$ (as in Figure~\ref{figurealternatingfans}) of $A_{n,m}$ $(m>0)$, with quiddity sequences $q$ and $\bar q$ respectively. Then
$$s_q=P_{2r}(n_1,m_1,n_2,m_2,\dots,m_{r-1},n_r,m_r)+P_{2r-2}(m_1,n_2,m_2,\dots,m_{r-1},n_r)=s_{\bar q}.$$
\end{cor}

We can easily make use of this to write a practical algebraic expression for the growth coefficient of a periodic infinite frieze of positive integers 
which does not contain determinants.

\begin{lem}\label{lem:growthcoefficient}
Let $\mathcal{F}$ be the outer frieze associated to a bridging triangulation $T$ (as in Figure~\ref{figurealternatingfans}) of $A_{n,m}$ $(m>0)$ with quiddity sequence $q$. Then
\begin{align*}
s_q&=2+\sum_{k=1}^{r}\left(\sum_{r\ge i_1\ge i_2> i_3\ge i_4>i_5\ge \dots >i_{2k-1} \ge i_{2k}\ge 1}m_{i_1}n_{i_2}m_{i_3}n_{i_4}\cdots m_{i_{2k-1}}n_{i_{2k}}\right)\\
&\quad +\sum_{k=1}^ {r-1}\left(\sum_{r\ge i_1> i_2\ge i_3>i_4\ge i_5 >\dots\ge i_{2k-1}> i_{2k}\ge 1}n_{i_1}m_{i_2}n_{i_3}m_{i_4}\cdots n_{i_{2k-1}}m_{i_{2k}}\right).
\end{align*}
\end{lem}

Note that all terms have even degree. If $m_{i_j}$ is followed by $n_{i_{j+1}}$, then $i_j\ge i_{j+1}$ and if $n_{i_j}$ is followed by $m_{i_{j+1}}$, then $i_j> i_{j+1}$.

\begin{proof}
By a straightforward inductive argument one can convince oneself, that
\begin{align*}
P_{2r}(n_1,m_1,\dots,n_r,m_r)&=1+\sum_{k\ge i_1\ge i_2\ge 1}m_{i_1}n_{i_2}+\sum_{k\ge i_1\ge i_2>i_3\ge i_4\ge 1}m_{i_1}n_{i_2}m_{i_3}n_{i_4}\\
&\quad+\dots+m_{r}n_{r}m_{r-1}n_{r-1}\cdots m_1n_1,
\end{align*}
and
\begin{align*}
P_{2r-2}(m_1,n_2,\dots,m_{r-1},n_r)&=1+\sum_{k\ge i_1> i_2\ge 1}n_{i_1}m_{i_2}+\sum_{k\ge i_1> i_2\ge i_3 >i_4\ge 1}n_{i_1}m_{i_2}n_{i_3}m_{i_4}\\
&\quad+\dots+n_{r}m_{r-1}n_{r-1}m_{r-2}\cdots n_2m_1.
\end{align*}
Hence with Corollary \ref{cor:growthcoefficient} the desired result follows.
\end{proof}

Here is an immediate corollary of Theorem \ref{cor:lem:growthcoefficient} and Lemma \ref{lem:growthcoefficient}.

\begin{cor}\label{cor:lem:growthcoefficient}
Given a periodic infinite frieze that is realisable in annulus but not in a punctured disc, we have $s_k>2$ for all $k>0$. 
\end{cor}

\begin{figure}[t]
\scalebox{.9}{\begin{tikzpicture}[font=\normalsize] 
  \matrix(m) [matrix of math nodes,row sep={1.5em,between origins},column sep={1.5em,between origins},nodes in empty cells]{
0&&0&&0&&0&&&&&&\\
&1&&1&&1&&1&&&&&\\
&&2&&2&&2&&5&&&&\\
&&&3&&3&&9&&9&&&\\
&&&&4&&13&&16&&13&&\\
&&&&&17&&23&&23&&17&\\
&&&&&&30&&33&&30&&72\\
&&&&&&&&&\node[rotate=-6.5,shift={(-0.034cm,-0.08cm)}]  {\ddots};&&\node[rotate=-6.5,shift={(-0.034cm,-0.08cm)}]  {\ddots};\\
};
\draw ($(m-6-6)+(0,0.6125cm)$) node[ellipse, minimum height=2cm,minimum width=0.75cm,draw,semithick,opacity=0.5] {};
\draw (m-5-4) node[shift={(-1.5cm,0cm)}]{$s_1=s_q=14$};

\draw[opacity=0.5,rounded corners,myblue] ($(m-3-3.north west)+(0cm,0.cm)$) -- ($(m-3-9.north east)+(0cm,0cm)$) --($(m-3-9.south east)+(0cm,0cm)$) --  ($(m-3-3.south west)+(0cm,0cm)$) -- cycle;
\end{tikzpicture}}
\caption{The $4$-periodic infinite frieze associated to the quiddity sequence $q=(2,2,2,5)$, with growth coefficient $s_1=14$.}\label{fig:friezeonefan}
\end{figure}
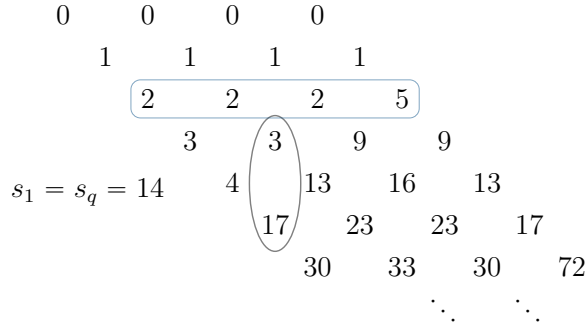

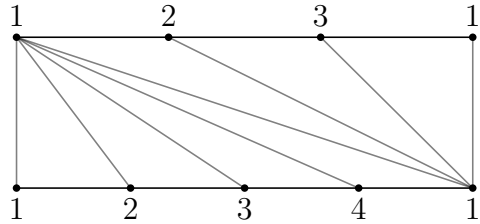
\begin{figure}[t]
\scalebox{1}{\begin{tikzpicture}[font=\normalsize]

\draw[semithick] (0,1) to (6,1);

\foreach \x in {0,2,4,6} {
   \begin{scope}[shift={(\x cm, 0 cm)}]
    \node (\x) at (0,1) [fill,circle,inner sep=1pt] {};
   \end{scope}
}

\draw (0,1) node [above] {$1$};
\draw (2,1) node [above] {$2$};
\draw (4,1) node [above] {$3$};
\draw (6,1) node [above] {$1$};

\draw[semithick] (0,-1) to (6,-1);

\foreach \x in {0,1.5,3,4.5,6} {
   \begin{scope}[shift={(\x cm, 0 cm)}]
    \node (\x) at (0,-1) [fill,circle,inner sep=1pt] {};
   \end{scope}
}

\draw (0,-1) node [below] {$1$};
\draw (1.5,-1) node [below] {$2$};
\draw (3,-1) node [below] {$3$};
\draw (4.5,-1) node [below] {$4$};
\draw (6,-1) node [below] {$1$};

\foreach \x in {0,1.5,3,4.5,6}{
\draw[semithick,opacity=0.5] (0 cm,1cm) to (\x cm,-1cm);
}

\foreach \x in {2,4,6}{
\draw[semithick,opacity=0.5] (\x cm,1cm) to (6 cm,-1cm);
}

\end{tikzpicture}}
\caption{The triangulation giving rise to the frieze in Figure~\ref{fig:friezeonefan}} \label{fig:friezeonefan-triang}
\end{figure}

\begin{ex}
Let $\mathcal{F}$ be the outer frieze associated to a (bridging) triangulation $T$ of $A_{n,m}$ $(m>0)$ with fundamental domain given by $r$ alternating fans ($r\ge1$). Let $q$ be the outer quiddity sequence.

\begin{inparaenum}[$(a)$]
\item If $r=1$, we obtain $s_q=2+mn$. This is exactly the case when we have a quiddity sequence of the form $(2,2,\dots,2,m+2)$. 
See Figures~\ref{fig:friezeonefan} and~\ref{fig:friezeonefan-triang}, where $n=4$ and $m=3$, yielding $s_1 = s_q = 14$.

\item If $r=2$, we have $s_q=2+m_2n_2+m_2n_1+m_1n_1+m_2n_2m_1n_1+n_2m_1.$

\item Let $r=3$. Then $s_q=2+m_3n_3+m_3n_2+m_3n_1+m_2n_2+m_2n_1+m_1n_1+m_3n_3m_2n_2
+m_3n_3m_2n_1+m_3n_3m_1n_1+m_3n_2m_1n_1+m_2n_2m_1n_1+m_3n_3m_2n_2m_1n_1+n_3m_2+n_3m_1+n_2m_1+n_3m_2n_2m_1.$ 
\end{inparaenum}
\end{ex}

\begin{cor}
For each integer greater or equal to $2$, there exists a periodic infinite frieze 
of positive integers for which the principal growth coefficient is given by this value.
\end{cor}

\section{Growth behaviour of periodic friezes of positive integers}\label{sec:analysis-periodic}

In this section, we analyse the sequence $(s_k)_{k\in\NN_0}$ of growth 
coefficients associated to a periodic frieze of positive integers or its lattice 
(with integer entries). 
We start by determining the growth coefficients 
for finite friezes, before considering the infinite cases. 
In order to determine the growth behaviour of 
the growth coefficients of periodic friezes (of positive integers), we consider in particular sequences 
$(s_k)_{k\in\NN_0}$ of integers with $s_0=2$, $s_1\in\ZZ$ and which satisfy the recursion
\begin{equation}\label{recursion:S}
s_{k+2} = s_1 s_{k+1} - s_k. 
\end{equation}

\begin{prop}\label{prop:finite-case-s}
Let $\mathcal F$ be a finite frieze of positive integers of order $n$. 
Then, $s_1\in \{-2,0,1\}$.  Furthermore, the value of $s_1$ is determined by the order of rotational symmetry of the associated triangulation (of an $n$-gon).
\end{prop}

\begin{proof}
Since $\mathcal F$ has order $n$, it is $n$-periodic and arises from a triangulation $T$ of an $n$-gon. In particular, the minimal period of $\mathcal F$ has to be an element of $\{n, \frac{n}{2}, \frac{n}{3} \}$. These correspond respectively to the cases where $T$ has no rotational symmetry, 180 degree rotational symmetry, and 120 degree rotational symmetry. We work in the lattice $\mathcal L=(m_{i,j})_{i,j}$ associated to $\mathcal F$.

If $n$ is the minimal period, then $s_1=m_{1,n}-m_{2,n+1}=-1 - 1=-2$. 

If the minimal period of $\mathcal F$ is $\frac{n}{2}$, then $-2=m_{1,n}-m_{2,n+1}=s_2$. Taking $k=0$ in Proposition~\ref{prop:relationvaluesA} (\ref{prop:relationvaluesA1}) yields $s_2 = s_1^2-s_0$, whereby $s_1^2 = 0$ and hence $s_1 = 0.$

Assume the minimal period is $\frac{n}{3}$. We have $s_3=-2$. We use 
Proposition~\ref{prop:relationvaluesA} (\ref{prop:relationvaluesA2}) to see that 
$-2=s_3=s_1^3 -3 s_1=s_1(s_1^2-3)$ holds. The only solutions 
to this polynomial are $s_1=1$ (repeated)  
and $s_1=-2$. In fact, by a simple inductive argument using matching numbers, it is easily established that we must have $s_1 = 1$ in this case. 
\end{proof}

\begin{cor}\label{cor:lm:finite-case-s}
If $\mathcal F$ is a finite frieze of positive integers, then its 
growth sequence is one of the following: 
\[
(2,-2,2,-2,2,-2,\dots),  \quad (2,0,-2,0,2,0,-2,\dots), \quad (2,1,-1,-2,-1,1,2,1,\dots ).
\]
The first of these growth sequences arises precisely when the associated triangulation 
has no rotational 
symmetry. The second and third respectively arise from triangulations with 180 degree and 120 degree rotational symmetry.
\end{cor}

\begin{proof} 
In view of Proposition~\ref{prop:finite-case-s} (and its proof), the result follows directly from Proposition~\ref{prop:relationvaluesA}.
\end{proof}

We note that the growth sequences appearing in Corollary~\ref{cor:lm:finite-case-s} also arise as special cases of Proposition~\ref{prop:PeriodicSolutions} below.

\begin{lem}\label{lem:properiescoefficents}
Let $\mathcal{F}$ be a periodic infinite frieze of positive integers. Then

\begin{inparaenum}[$(1)$]
\item\label{lem:properiescoefficents1} $s_k\ge2$ for all $k\ge 0$,

\item\label{lem:properiescoefficents2} $s_k\le s_{k+1}$ for all $k\ge 0$.
\end{inparaenum}

\noindent Furthermore, the inequalities in $(2)$ are strict when $s_1>2$.
\end{lem}

\begin{proof}
\begin{inparaenum}[$(1)$]
\item In the case where $\mathcal{F}$ arises from a triangulation of a punctured disc, we have from Remark~\ref{rem:growthcoefficients} that $s_k=2$, for all $k$. Otherwise, we have $s_k>2$ for all $k$ from Corollary~\ref{cor:lem:growthcoefficient}.

\item Let $k=1$. By Proposition~\ref{prop:relationvaluesA}~(\ref{prop:relationvaluesA1}) we have $s_2=s_1^2-2$ and from (\ref{lem:properiescoefficents1}) we have $s_1\ge 2$, hence $s_2\ge s_1$. Then, by  
Proposition~\ref{prop:relationvaluesA}~(\ref{prop:relationvaluesA1}) and induction, i.e. $s_{k+1}=s_1s_k-s_{k-1}\ge s_1s_k-s_k=s_k(s_1-1)$, we obtain $s_{k+1}\ge s_k$
for $k\ge 1$, see also Proposition~\ref{prop:S} below. The final sentence of the statement is clear.
\end{inparaenum}
\end{proof}

\begin{prop}\label{prop:puncturedDisk}
Let $\mathcal{F}$ be a periodic infinite frieze of positive integers. 
The following are equivalent:

\begin{inparaenum}[$(a)$]
\item There exists $k>0$ such that $s_k=2$. 

\item $s_k=2$ for all $k\ge 0$.

\item $\mathcal F$ arises from a triangulation of a punctured disc.
\end{inparaenum}
\end{prop}

\begin{proof}
We first show that $(a)$ and $(b)$ are equivalent. Assume $s_t=2$ for some $t>0$. By Lemma~\ref{lem:properiescoefficents}~(\ref{lem:properiescoefficents1}) and (\ref{lem:properiescoefficents2}), we have $2\le s_1\le s_2\le\dots\le s_t=2$, so $s_k=2$ for all $1\le k\le t$. With Proposition~\ref{prop:relationvaluesA} we can deduce inductively that $s_k=2$ for all $k> t$. The other implication is trivial.

To see that $(b)$ implies $(c)$ is an immediate consequence of Corollary~\ref{cor:lem:growthcoefficient}, and that $(c)$ implies $(b)$ is stated in Remark \ref{rem:growthcoefficients}.
\end{proof}

\begin{cor}
\label{cor:annulus}
Let $\mathcal{F}$ be a periodic infinite frieze of positive integers. 
The following are equivalent:

\begin{inparaenum}[$(a)$]
\item There exists $k>0$ such that $s_k>2$. 

\item $s_k>2$ for all $k> 0$.

\item $\mathcal F$ does not arise from a triangulation of a punctured disc.
\end{inparaenum}
\end{cor}

\begin{thm}[Dynamics of periodic friezes of positive integers]\label{thm:growth}
Let $\mathcal F$ be an $n$-periodic frieze of positive integers with lattice $\mathcal L$. 
Then the entries in any diagonal of $\mathcal L_+$ must have exactly one of the 
three behaviours: (a) they form periodic sequences; (b) they grow linearly (when jumping in steps of $n$ positions); (c) their growth is asymptotically exponential, 
as specified in Proposition~\ref{rem:growth} below.  
\end{thm}
\begin{proof}
The lattice $\mathcal L$ arises from a finite or infinite frieze of positive integers. In the 
finite case, it is clear that each diagonal of the lattice $\mathcal L_+$, and also $\mathcal L$, forms a periodic sequence. 

So, assume $\mathcal F$ is infinite.
We have seen that $s_k=2$ for all $k$ if and only if $\mathcal F$ arises from a triangulation of 
a punctured disc (Proposition~\ref{prop:puncturedDisk}). This is equivalent to linear growth in the diagonals, for sequences of entries separated by $n$ positions.
Finally, statement (c) follows from Proposition~\ref{rem:growth} below. 
\end{proof}

\begin{prop}[Asymptotic exponential growth of superlinear periodic infinite friezes of positive integers]\label{rem:growth}
Let $(s_k)_{k\in\NN_0}$ be the sequence of growth coefficients of a periodic infinite frieze of positive integers not arising from a triangulation of a punctured disc.

Then, the sequence $(s_k)_{k\in\NN_0}$ grows monotonically and asymptotically exponentially with the rate 
$$
S := \frac{1 + \sqrt{s_1^2-4}}{2},\qquad\text{where}\quad S\in(1,\infty) \quad\text{for}\quad s_1\in(2,\infty),
$$
in the sense that for every $0<\delta$ and every sufficiently large index $\kappa=\kappa(\delta)$ 
satisfying  $S^{-2(\kappa-1)}(s_1-S-1)\le \delta$,
the following asymptotically exponential growth holds 
\begin{equation}\label{superlinear}
s_{\kappa}\, S^l < {s_{\kappa+l}} < s_{\kappa} (S+\delta)^l \qquad \text{for all} \quad 1\le l\to+\infty.
\end{equation}
Thus, $S$ is the exponentially asymptotic exponential growth rate (i.e.\ the asymptotic exponential growth is approximated exponentially fast) of the sequence $(s_k)_{k\in\NN_0}$.  
\medskip

Moreover, the frieze entries asymptotically obey the same asymptotically exponential growth when considered in blocks of $n$ rows. 
In particular, it follows from Corollary~\ref{cor:s-k-relation}, i.e. from $m_{i,j+2kn}=s_k m_{i,j+kn}-m_{i,j}$ by recursively setting  $k=2^{l-k'}$ that 
$$
m_{i,j+2^{l+1} n} = \prod_{k'=0}^{l} s_{2^{l-k'}}\, m_{i,j+n} - \sum_{m'=0}^{l} \prod_{k'=0}^{l-1-m'} s_{2^{l-k'}}\, m_{i,j},
$$
which shows in the asymptotic behaviour $s_{2^{l-k'}} \sim S^{2^{l-k'}}$
$$
m_{i,j+2^{l+1} n} \sim S^{2^{l+1}-1}\, m_{i,j+n} - 
\sum_{m'=0}^{l} S^{2^{l-m'}-1}\, m_{i,j}.
$$
Thus, since for any frieze entries $m_{i,J}$
between the $2^{l+1}$-th and the $2^{l+2}$-th block of $n$ rows of the lattice $\mathcal{L}_+$, we have $J=j+2^{l+1}n$ with 
$j=\tilde j + 2^{l}n$, it follows recursively that 
$m_{i,J}=m_{i,j+2^{l+1}n} = O(S^{2^{l+1}})=O(s_{2^{l+1}})$.
\end{prop}

\begin{proof}
First, we note that $s_1>s_0=2$ yields $x_1 \coloneq \frac{s_2}{s_1}=s_1-\frac{s_0}{s_1}>S$ (more precisely, it is easy to verify that $S< s_1 - 1 < x_1$ for all $s_1>2$).

Then, Proposition \ref{rem:growth} follows from case (EG) of Proposition \ref{prop:S} below, which provides a more general analysis of the recursion \eqref{recursion:S}. 
More precisely, the exponentially asymptotic exponential growth \eqref{superlinear} follows 
from case (EG) of Proposition~\ref{prop:S} specialised to values $x_1>s_1-1>S$, and thus from the estimate \eqref{rightconv}
of Lemma \ref{lem:x}, which also implies that $S<x_k$ and $2<s_k$ for all $k\ge1$. Finally, the growth behaviour of the 
frieze entries follows in an elementary manner via iteration by using 
$\sum_{k=0}^l {2^{l-k}} = 2^{l+1}-1$.
\end{proof}

\section{Growth behaviour of real-valued sequences and tame friezes} \label{sec:analysis-recursion}
 
In this section, we generalise the framework of growth coefficients for tame real periodic friezes and their lattices.
More precisely, we consider sequences $(r_k)_{k\in\NN_0}$, which are defined via the recursion \eqref{recursion:S}, yet start from arbitrary initial values, i.e.  
\begin{equation}\label{recursion:r}
r_{k+2} = r_1 r_{k+1} - r_k, \qquad r_1,r_0 \in \RR, k\geq 0.
\end{equation} 

A comprehensive analysis of the behaviour of the sequences $(r_k)_{k\in\NN_0}$ defined via recursion
\eqref{recursion:r} is provided below by Proposition~\ref{prop:S} for $|r_1|\ge2$ and Proposition~\ref{prop:S2} for $|r_1|<2$. 
The results obtained are more general than what is required for the growth coefficients of infinite 
integer- or real-valued friezes and also go beyond the relationship between the 
(growth coefficient like) sequences $(s_k)_{k\in\NN_0}$ with $s_0=2$ and Chebyshev polynomials in $s_1$, see e.g.\ Example~\ref{ex:skintermsofs1}. Please note that for $r_0\neq 2$, the sequences $(r_k)_{k\in\NN_0}$ do not rescale to Chebyshev polynomials in $r_1$.
\medskip

One key result of this section is Theorem~\ref{thm:two-real-sequences} on the existence of generalised recursion formulas, 
which implies an interesting relation between the existence of $p$-periodic solutions of recursion \eqref{recursion:r}
and the condition $s_p=2$, where $s_p$ is the $p$-th entry of the specific solution sequence $(s_k)_{k\in\NN_0}$ of recursions \eqref{recursion:r} with $s_0=2$ and $s_1=r_1\in\RR$.
\medskip

At first, we point out that by applying the rescaling $r_k = r_1 t_k$ for $r_1\neq0$ and  $k\ge 0$, the rescaled recursion, 
\begin{equation*}\label{recursion:rscaled}
t_{k+2} = r_1 t_{k+1} - t_k, \qquad t_1=1,t_0=\frac{r_0}{r_1} \in \RR
\end{equation*}
separates the role of $r_1$ as parameter of the recursion from the role of $r_1$ as initial value, which is now normalised to $t_1=1$ without loss of generality. 
The fact that this rescaled recursion  
is subject to only one arbitrary initial value $t_0$ suggests that the two-stage recursion \eqref{recursion:r} 
can be transformed into a  one-stage recursion: Indeed, by denoting $x_{k+1} := \frac{r_{k+2}}{r_{k+1}}$, we find
\begin{equation}\label{recursion:xx}
x_{k+1} = r_1 - \frac{1}{x_{k}} =: f(r_1,x_{k}),\qquad x_1 = \frac{r_2}{r_1} = \frac{r_1^2-r_0}{r_1}\in \RR.
\end{equation} 
Moreover, in case $x_k=0$ for some $k\ge 1$, then we uniquely extend the recursion \eqref{recursion:xx} by setting 
\begin{equation}\label{extension}
x_{k}=0 \quad\Longrightarrow\quad
x_{k+1}:=-\infty,\qquad x_{k+2}:=r_1,
\end{equation}
and continue \eqref{recursion:xx} accordingly. 

The extension \eqref{extension} corresponds exactly to the behaviour of recursion 
\eqref{recursion:r}, where $r_{k+1}=0$ (which is equivalent to $x_k=0$) yields 
\begin{equation*}
r_{k+2}=-r_k, \quad 
r_{k+3}=r_1 r_{k+2}\quad\text{and thus} 
\quad x_{k+2} = \frac{r_{k+3}}{r_{k+2}}= r_1.
\end{equation*}
Moreover, the sign of $x_{k+1}:=-\infty$ reflects the fact that the sequence 
$r_k$ changes sign when passing through zero, i.e. 
$\mathrm{sign}(r_k r_{k+2}) = -1$.

\begin{prop}[Growth behaviour of the recursion \eqref{recursion:r} for $|r_1|\ge2$]\label{prop:S}\hfill\\
Let $|r_1|\ge2$. Then, the recursion 
\begin{equation*}
r_{k+2} = r_1 r_{k+1} - r_k, \qquad |r_1|\ge2,r_0 \in \RR.
\end{equation*}
yields sequences $(r_k)_{k\in\NN_0}$ satisfying the following cases of linear and superlinear growth:
\begin{itemize}
\item[(LG)] (Linear/Constant Growth) If $r_1=2$, then a straightforward induction proves 
$$
r_{k+1} = (k+1) r_1 - k r_0 = k (r_1-r_0) + r_1.
$$
In particular, the only one-periodic, constant pattern of the recursion 
\eqref{recursion:r} (see also Proposition~\ref{prop:PeriodicSolutions} below) has the initial values $r_0=r_1=2$, i.e.\ $r_k = 2$ for all $k\in\NN_0$.

Moreover, the only integer-valued sequences $(r_k)_{k\in\NN_0}$ for which $r_k=0$ is possible for some $k \geq 2$ (and thus the extension rule \eqref{extension} is necessary) are found in the special case $r_2=1$, e.g.\
$(r_0,r_1,r_2,r_3,r_4,r_5,\ldots) = (3,2,1,0,-1,-2,\ldots)$.
\medskip
\item[(EG)] (Exponential Growth) If $r_1>2$ and $r_0\in \RR$, then the behaviour of the two-stage recursion \eqref{recursion:r} follows from the one-stage recursion \eqref{recursion:xx} with the extension rule \eqref{extension}.
In particular, Lemma~\ref{lem:x} implies 
that the recursion \eqref{recursion:r} 
grows exponentially in the following sense: For every $0<\varepsilon<S-1$ 
with $S$ given in \eqref{xinfty1}, i.e. 
$$
S := \frac{1 + \sqrt{r_1^2-4}}{2},\qquad\text{where}\quad S\in(1,\infty) \quad\text{for}\quad r_1\in(2,\infty).
$$
there exists an index $K=K(\varepsilon)$, sufficiently large, such that 
\begin{equation*}
(S-\varepsilon)^l \le \frac{r_{K+l}}{r_K} \le  (S+\varepsilon)^l \qquad \text{as} \qquad l\to\infty .
\end{equation*}
Note, that the exponential growth stated in 
the above inequalities holds independently of the sign of $r_K$, and that it is possible 
that the sequence $(r_k)_{k\in\NN_0}$ underwent a sign change prior to the index $K$.

\medskip 
 
The only integer-valued sequences $(r_k)_{k\in\NN_0}$ for which $r_k=0$ is possible (and thus the extension rule \eqref{extension} is necessary) are found in the special cases $r_2=1$, e.g.\ for $r_1=3$, we find
$(r_0,r_1,r_2,r_3,r_4,r_5,\ldots) = (8,3,1,0,-1,-3,\ldots)$.
\medskip
\item[(S)] (Sign Alternating Growth) If $r_1\le -2$ and $r_2\in \RR$, then the following modified transformation 
of the two-stage recursion \eqref{recursion:r} into the one-stage recursion \eqref{recursion:xx} (with the extension rule \eqref{extension}) applies: By denoting $\tilde x_{k+1} := -\frac{r_{k+2}}{r_{k+1}}$, we have 
\begin{equation*}
\tilde x_{k+1} = -r_1 - \frac{1}{\tilde x_{k}} =: f(|r_1|,\tilde x_{k}),\qquad \tilde x_1 = -\frac{r_2}{r_1}\in \RR.
\end{equation*}
Thus, the results of cases {(LG)} and {(EG)} carry over accordingly and lead to sign alternating, linear or superlinear growing  sequences $(r_k)_{k\in\NN_0}$.

In the \emph{special case $r_1=-2$}, Proposition~\ref{prop:PeriodicSolutions} implies that the only possible 
periodic pattern is two-periodic with $r_0=2$.
In this case, we have
$$
(r_0,r_1,r_2,r_3,s_4,\ldots) = (2,-2,2,-2,2,\ldots), 
$$
which arises from any triangulation of a polygon which has no rotational symmetry, i.e. from the quiddity sequence (1,4,1,2,2,2),
and corresponds to $\tilde x_{k+1}=\tilde x_1=1 $ for all $k\in\NN_0$ with $\tilde x_1=-\frac{-2}{2}$.
\end{itemize}
\end{prop}
\begin{proof} The proof of item (LG) is a straightforward induction, while the proof of (EG) follows from Lemma~\ref{lem:x} 
stated in the Appendix~\ref{appb}. Finally, (S) then follows as explained in the statement.
\end{proof}

Proposition~\ref{prop:S} provides an essentially  complete analysis of the recursion \eqref{recursion:r},
for the cases $|r_1|\ge2$, based on the results of Lemma~\ref{lem:x}.
\medskip

For $|r_1|<2$, Lemma \ref{lem:x} implies that the one-stage recursion \eqref{recursion:x} has no fixed point and that the dynamics of the growth coefficient recursion \eqref{recursion:r}  must feature periodic solutions or 
possibly more complicated behaviour (like almost-periodic or even chaotic sequences). 
A complete analysis of the recursion \eqref{recursion:r} in this case may therefore be very hard.

The following theorem exhibits an interesting relationship between any real-valued sequence $(r_k)_{k \in \NN_0}$ satisfying \eqref{recursion:r} and the specific associated sequence $(s_k)_{k \in \NN_0}$, whose initial values are $s_0 = 2$ 
and $s_1 = r_1$, and which also satisfies \eqref{recursion:r}. Namely, the elements of the latter sequence appear as coefficients in generalised recursion formulas involving the elements of the former.

\begin{thm}[Generalised recursion formulas]\label{thm:two-real-sequences}\hfill\\
Consider a real valued sequence $(r_k)_{k\in \NN_0}$ 
such that the recursion \eqref{recursion:r} is satisfied.
Let $(s_k)_{k\in \NN_0}$ be the sequence satisfying \eqref{recursion:r} with the initial values $s_0=2$, $s_1=r_1\in\RR$.
Then, for all $p\ge0$ and $k\ge 0$, we have 
\[
r_{k+2p}=s_{p} r_{k+p}-r_{k}.
\]
\end{thm}

\begin{proof}
We show the claim using induction on $p$. For $p=0$, the claim requires $r_{k}=2 r_{k}-r_{k}$, which is true. If $p=1$, we get $r_{k+2}=s_1 r_{k+1}-r_{k}$, which holds by the recursion \eqref{recursion:r} since $s_1 = r_1$.

Now assume the claim holds for some $p\ge 1$, and also for $p-1$. We have to deduce that $r_{k+2(p+1)}-s_{p+1} r_{k+(p+1)}+r_{k}=0$. First, by the recursion \eqref{recursion:r}, i.e. $r_{k+2(p+1)}=r_1r_{k+2p+1}-r_{k+2p}$, and by the induction hypothesis $r_{k+1+2p}=s_pr_{k+1+p}-r_{k+1}$, we get  
$$
r_{k+2(p+1)}=
r_1 s_p r_{k+1+p}-r_1 r_{k+1}-r_{k+2p}. 
$$
Then, by using \eqref{recursion:r} again, i.e. 
$-r_1r_{k+1} = - r_{k+2}-r_k$ and $- s_{p+1} = -r_1s_{p}  +s_{p-1}$, we compute
$$
r_{k+2(p+1)}-s_{p+1} r_{k+(p+1)}+r_{k}= 
-r_{k+2} - r_{k+2p} + s_{p-1} r_{k+p+1}=0,
$$
where the last identity holds again by the induction hypothesis
since $s_{p-1} r_{k+p+1}= s_{p-1} r_{k+2+(p-1)}=r_{k+2+2(p-1)}+r_{k+2}$.
This completes the proof.
\end{proof}

\begin{rem}
Theorem~\ref{thm:two-real-sequences} generalises the special case when $r_0 = 2$ (i.e.\ $r_k = s_k$ for all $k \geq 0$) for which the specialised recursion formulas $s_{k+2p}=s_{p} s_{k+p}-s_{k}$
correspond to a well-known property of Chebyshev polynomials. 
Moreover, we remark that in this special case, the generalised recursions $s_{k+2p}=s_{p} s_{k+p}-s_{k}$ for $p\ge1$ enable the derivation of further explicit expressions, e.g.\ for $s_{2^{l}}$ for $l\in\NN_0$.
\end{rem}

As a consequence of Theorem~\ref{thm:two-real-sequences}, 
we can, for any $p\ge1$, define 
$$
y_{l+p}:=\frac{r_{l+2p}}{r_{l+p}}\quad \text{and}\quad  
y_{l}:=\frac{r_{l+p}}{r_{l}},
$$ 
and observe that 
\begin{equation}\label{recursion:yp}
y_{l+p}=s_p-\frac{1}{y_{l}}= f(s_p,y_{l}), \qquad \text{ for all }\quad 1\le p, \quad 0\le l \le p-1, 
\end{equation}
holds, and thus that for all $0\le l \le p-1$ the subsequences $(y_{l}, y_{l+p}, y_{l+2p},\ldots)$
obey the one-stage recursion \eqref{recursion:xx} 
with modified parameter $r_1=s_1\to s_p$ for all $p\ge1$. 
In particular, we recall that 
\eqref{recursion:yp} is therefore a generalisation of the recursion \eqref{recursion:xx}, which we recover in the case $p=1$,
with $y_{l}=x_l$. The multi-stage recursion \eqref{recursion:yp} yields the following Corollary 
on periodic subsequences and sequences.

\begin{cor}[Periodic subsequences and periodic sequences]\label{persol}\hfill\\
Consider, as in Theorem~\ref{thm:two-real-sequences}, a real-valued sequence $(r_k)_{k\in \NN_0}$ satisfying the recursion \eqref{recursion:r}, together with the sequence $(s_k)_{k\in \NN_0}$ with the initial values $s_0=2$, $s_1=r_1\in\RR$. Let $1\le p\in\NN$. 
Then, for $0\le l \le p-1$, the sequence $(r_k)_{k\in \NN_0}$ has a $p$-periodic subsequence $(r_{l+np})_{n\in \NN_0}$, i.e. 
$$
r_{l+(n+1)p} = y_{l+np}\, r_{l+np} = r_{l+np}, 
\qquad \forall n\in\NN_0, 
$$
if and only if $s_p=2$ and $y_l=1$. 

Moreover, if $y_0=1$ and $y_1=1$, then $y_l=1$ for all 
$0\le l \le p-1$ and the sequence $(r_k)_{k\in \NN_0}$
is consequently $p$-periodic provided that $s_p=2$.
\end{cor}

\begin{proof}
First, the analysis of the recursion \eqref{recursion:yp} in Lemma~\ref{lem:x} in Appendix~\ref{appb} implies for all 
$0\le l \le p-1$ that $y_{l+np}=1$ is a fixed point of \eqref{recursion:yp} if and only if $s_p=2$ and $y_l=1$.
Then, given that $y_{l}=1$ and $y_{l+1}=1$,  it follows by induction that $r_{p+l+2} = r_{1} r_{p+l+1} -r_{p+l} = 
r_{1} r_{l+1} -r_{l} = r_{l+2}$, which implies $y_{l+2}=1$.
\end{proof}

\begin{rem}
In fact, explicit calculations will show that for all the examples of periodic solutions of \eqref{recursion:r} exhibited in Proposition~\ref{prop:PeriodicSolutions} below, the conditions $s_p=2$ and $y_0=1$ imply also $y_1=1$.
Moreover, for $p\ge3$, it is even true in these cases that $s_p=2$ implies $y_0=1$ and $y_1=1$, which is the reason for 
three-, four-, and six-periodic solutions existing for all $r_0\in\RR$. However, we were unable to turn 
these observations into a general theorem.
\end{rem}

Corollary~\ref{persol} now enables us to prove the following Proposition~\ref{prop:PeriodicSolutions}, 
which characterises all one-, two-, three-, four- and six-periodic solutions of the  
recursion \eqref{recursion:r}. In particular, we are able to identify all possible periodic solutions of  
\eqref{recursion:r} for the integer values $r_1=0,\pm1,\pm2$. For higher periods $p$, we point 
out that the condition $s_p=2$ (despite being a 
polynomial equation of order $p$) can be solved due to the relationship between the sequence $(s_k)_{k\in \NN_0}$
and the properties of extremal points of Chebyshev polynomials. More precisely, we have that 
\begin{equation*}
s_p = 2 \cos (p \arccos({r_1}/{2})), \qquad \text{if}\quad |r_1|\le 2,
\end{equation*}
and the well-known maxima of Chebyshev polynomials on $[-1,1]$ imply that 
\begin{equation}\label{maximalpoints}
s_p=2\quad \iff\quad r_1 = 2 \cos\left(\frac{2k}{p}\pi\right), \qquad k=0,\ldots,\left\lfloor\frac{p}{2}\right\rfloor.
\end{equation}
Please note that not all values of $r_1$ for which $s_p=2$ holds, give rise to genuinely $p$-periodic behaviour. 
By taking $k=0$ in \eqref{maximalpoints}, for instance, the following Proposition~\ref{prop:PeriodicSolutions} will show 
that $r_1=2$ allows only a one-periodic solution, which is of course also $p$-periodic. Similarly, all divisors of $p$ 
yield non-genuinely $p$-periodic solutions.

\begin{prop}[Characterisation of periodic solutions of \eqref{recursion:r} of small periods]
\label{prop:PeriodicSolutions} \hfill\\
Consider the recursion 
\begin{equation*}
r_{k+2} = r_1 r_{k+1} - r_k, \qquad r_1,r_0 \in \RR.
\end{equation*}
Then, 

\begin{itemize}[topsep=5pt, leftmargin=7mm]
\item the only possible one-periodic sequence is the constant sequence $(r_0,r_1,r_2,\ldots) = (2,2,2,\ldots)$.
In turn, given $r_1=2$, the only possible periodic sequence requires also $r_0=2$, which yields 
the same constant sequence $r_{k} = 2, \forall k\in\NN_0$.

\item the only possible two-periodic sequence is $(r_0,r_1,r_2,r_3,\ldots) = (2,-2,2,-2,\ldots)$. 
In turn, given $r_1=-2$, the only possible periodic sequence requires $r_0=2$, which yields 
the same two-periodic sequence $r_{k} = (-1)^k 2, \forall k\in\NN_0$.

\item there exists, for all $r_0\in\RR$, a one-parameter family of three-periodic sequences: $(r_0,r_1,r_2,r_3,\ldots) = (r_0,-1,1-r_0,r_0,\ldots)$. In turn, given $r_1=-1$, any sequence satisfying \eqref{recursion:r} is such a three-periodic sequence, i.e.  
$$
r_{0+3l} = r_0, \quad 
r_{1+3l} = -1, \quad 
r_{2+3l} = 1-r_0, \qquad \forall l\in\NN_0.
$$
\item there exists, for all $r_0\in\RR$, a one-parameter family of four-periodic sequences: 
$(r_0,r_1,r_2,r_3,r_4,\ldots) = (r_0,0,-r_0,0,r_0,\ldots)$. 
In turn, given $r_1=0$, any sequence satisfying \eqref{recursion:r} is such a four-periodic sequence, i.e.  
$$
r_{2l} = (-1)^{l} r_0,     \quad r_{2l+1} = 0, \qquad l\in \NN_0.
$$
\item there exists, for all $r_0\in\RR$, a one-parameter family of six-periodic sequences: 
$$
\qquad(r_0,r_1,r_2,r_3,r_4,r_5,r_6,\ldots) = (r_0,1,1-r_0,-r_0,-1,-1+r_0,r_0,\ldots).
$$
In turn, given $r_1=1$, any sequence satisfying \eqref{recursion:r} is such a six-periodic (and three-anti-periodic) sequence, i.e. 
$$
r_{0+3l} = (-1)^l r_0, \quad 
r_{1+3l} = (-1)^l , \quad 
r_{2+3l} = (-1)^l (1-r_0), \qquad \forall l\in\NN_0.
$$
\end{itemize}
\end{prop} 

\begin{proof}[Proof of Proposition~\ref{prop:PeriodicSolutions}]
First, we recall the one-stage recursion \eqref{recursion:xx}, i.e. 
$$
x_{k+1} = r_1 - \frac{1}{x_k} \qquad \text{and} \qquad
x_{k} = \frac{1}{r_1 - x_{k+1}}.
$$
The following cases follow from elementary considerations and 
Theorem~\ref{thm:two-real-sequences}.
\medskip\\
\noindent\emph{One-periodic solutions of \eqref{recursion:r}}: We recall recursion \eqref{recursion:xx}, or equally recursion \eqref{recursion:yp} for $p=1$, and require
$$
\qquad 
r_{k+1} = {x_k} r_k = r_k, \quad\iff\quad 
x_k=1, \quad \forall k\in\NN, 
$$
which is equivalent to 1 being a fixed point of \eqref{recursion:xx}, which is in turn equivalent to $r_1=2$ and $x_1=1$ due to Lemma~\ref{lem:x} in Appendix~\ref{appb}. 
Since  $2=x_1 r_1= r_2=r_1^2-r_0$, these two conditions imply $r_0=2$. Thus, as already found in Proposition~\ref{prop:S}, the only possible one-periodic pattern is
$
(r_0,r_1,r_2,r_3,\ldots) = (2,2,2,2,\ldots). 
$
\medskip

\noindent\emph{Two-periodic solutions of \eqref{recursion:r}}:  
By recalling the multi-stage recursion \eqref{recursion:yp} for $p=2$ and $s_2=r_1^2-2$, i.e. 
\begin{equation}\label{recursion:y}
y_{k+2} = (r_1^2-2) - \frac{1}{y_{k}}=f(r_1^2-2,y_{k}),
\end{equation}
we find (again by Lemma~\ref{lem:x}) that \eqref{recursion:y} has a fixed point $y_{k+2}=y_{k}=1$ iff $r_1^2-2=2$, i.e. 
$r_1=\pm 2$ and $y_0=1$. Similarly, a fixed point $y_{k+3}=y_{k+1}=1$ occurs iff $r_1^2-2=2$, i.e. 
$r_1=\pm 2$ and $y_1=1$.
Moreover, we calculate 
$r_0 = y_0 r_0 = r_2 = 4 - r_0 $, which yields $r_0=2$. Also, 
$r_3 = r_1 (r_1^2-r_0-1)$, which implies that $y_1=1$ holds 
also provided $r_0=2$. 
Thus, for $r_1=2$, we recover the above one-periodic pattern,  while for $r_1=-2$, we have the stated two-periodic pattern $(2,-2,2,-2,\ldots)$.

\medskip

\noindent\emph{Three-periodic and three-anti-periodic solutions of \eqref{recursion:r}}: 
By recalling the multi-stage recursion \eqref{recursion:yp} for $p=3$ and $s_3=r_1(r_1^2-3)$, i.e. 
\begin{equation}\label{recursion:z}
\qquad\
y_{k+3} = r_1(r_1^2-3) - \frac{1}{y_{k}}
\end{equation}
and Corollary~\ref{persol}, we observe that $(r_k)_{k\in\NN_0}$ is a three-periodic pattern iff
$r_1(r_1^2-3)= 2$ and $y_0=y_1=1$. 
Moreover, $r_1(r_1^2-3)= 2$ has the solutions $r_1=-1,2$.
For $r_1=2$, the above arguments recover the one-periodic pattern $(2,2,2,2,\ldots)$. For $r_1=-1$, we calculate $r_3=r_0$
which yields $y_0=1$ independently of $r_0$. Moreover, $r_4=-1=r_1$ implies $y_1=1$ holds (independently of $r_0$). 
Thus, for any $r_0\in\RR$, we have three-periodic patterns 
$(r_0,-1,1-r_0,r_0,-1,1-r_0,\ldots)$.

More generally, all fixed points $y$ of the 
three-step recursion
\eqref{recursion:z} must satisfy 
$$
y^2-r_1(r_1^2-3) y +1=0, \quad \Longrightarrow \quad
y = \frac{r_1(r_1^2-3)}{2} \pm\frac12 |r_1^2-1|\sqrt{r_1^2-4}.
$$
Three-periodic solutions correspond to $y=1$, which is equivalent to $r_1=-1$. Moreover, 
for $r_1=1$, we have $y=-1$, which leads to three-anti-periodic solutions and thus, six-periodic solutions provided that $y_0=y_1=-1$ (see below).
\medskip

\noindent\emph{Four-periodic solutions of \eqref{recursion:r}}:
By recalling the multi-stage recursion \eqref{recursion:yp} for $p=4$ and $s_4=r_1^4-4r_1^2+2$, i.e. 
\begin{equation}\label{recursion:zeta}
\qquad\
y_{k+4} = r_1^4-4r_1^2+2 - \frac{1}{y_{k}},
\end{equation}
we see that $(r_k)_{k\in\NN_0}$ is a four-periodic pattern iff
$r_1^4-4r_1^2+2= 2$ and $y_0=y_1=1$. 
Moreover, $r_1^4-4r_1^2+2= 2$ has the solutions 
$r_1=0,\pm2$. While $r_1=\pm2$ leads to the above 
one- and two-periodic patterns, it is straightforward for $r_1=0$ to deduce four-periodic patterns for any $r_0\in\RR$:
$r_{2k+1} = 0,     \
r_{2k} = (-1)^{k} r_0, \ k\in \NN_0$.

\medskip

\noindent\emph{Six-periodic solutions of \eqref{recursion:r}}: 
By recalling the multi-stage recursion \eqref{recursion:yp} for $p=6$ and $s_6=r_1^2(r_1^4-6r_1^2+9)-2$, i.e. 
\begin{equation}\label{recursion:zzz}
\qquad\
y_{k+6} = r_1^2(r_1^4-6r_1^2+9)-2 - \frac{1}{y_{k}},
\end{equation}
we see that $(r_k)_{k\in\NN_0}$ is a six-periodic pattern iff
$r_1^2(r_1^4-6r_1^2+9)-2= 2$ and $y_0=y_1=1$. 
Moreover, $r_1^2(r_1^4-6r_1^2+9)-2= 2$ has the solutions 
$r_1=\pm1,\pm2$. For $r_1=\pm2$, this leads to the above one- and two-periodic patterns. For $r_1=-1$, we recover the above three-periodic family of patterns. Finally, for $r_1=1$, we recover precisely the six-periodic family of patterns that arose above as three-anti-periodic patterns. These six-periodic patterns have the form 
$(r_0,1,1-r_0,-r_0,-1,-1+r_0,r_0,1,1-r_0,\ldots)$, with $r_0 \in \RR$.
\end{proof}

\begin{ex}
A source of one-periodic real friezes are regular convex polygons. Let $P$ be the regular convex $n$-gon with sides of length one: the length of the shortest diagonal in $P$ serves as the entry $a$ for the quiddity 
sequence $(a)$. One can check that the diamond rule then amounts to 
calculating the lengths of the different diagonals in the polygon via the Ptolemy rule. 
In the case of the square, $a=\sqrt{2}=r_1$, in the 
case of the hexagon, $a=\sqrt{3}=r_1$. The former leads to a four-anti-periodic sequence 
$(r_k)_k$, the latter to 
a six-anti-periodic sequence $(r_k)_k$, both with $r_0=2$. We can take this further and describe any periodic sequence $(r_k)_k$ with arbitrary $r_0 \in \RR$ and $r_1 = \pm \sqrt{2}$ or $r_1 = \pm \sqrt{3}$:
\begin{itemize}[topsep=5pt, leftmargin=7mm]
\item If $r_1=\pm\sqrt{2}$, then all sequences 
$(r_k)_{k\in\NN_0}$ satisfying \eqref{recursion:r} are four-anti-periodic and thus eight-periodic. 
\item If $r_1=\pm\sqrt{3}$, then all sequences 
$(r_k)_{k\in\NN_0}$ satisfying \eqref{recursion:r} are six-anti-periodic and thus twelve-periodic. 
\end{itemize}
Indeed, four-anti-periodic sequences for $r_1=\pm\sqrt{2}$ can be studied by checking for a fixed point $y$ of 
\eqref{recursion:zeta}, i.e.  
$$
y^2 - (r_1^4-4r_1^2+2) y +1 =0  \quad \Longrightarrow\quad
y = \frac{r_1^4-4r_1^2+2}{2} \pm \frac12 |r_1||r_1^2-2|\sqrt{r_1^2-4}.
$$
We recall (and also observe readily from the above formula) that the fixed point $y=1$ is equivalent to $r_1=0,\pm2$ and leads to the four-periodic sequences stated in Proposition~\ref{prop:PeriodicSolutions}. 
However, for $r_1=\pm \sqrt{2}$, we find the fixed point $y=-1$, which leads to four-anti-periodic patterns. In particular, one can verify that for any $r_0\in\RR$, $r_4=-r_0$ and $r_5=-r_1=\mp\sqrt{2}$ must then hold, and Corollary~\ref{persol} yields, for any $r_0\in\RR$, the four-anti-periodic pattern 
$$
(r_0,\pm\sqrt{2},2-r_0,\pm\sqrt{2}(1-r_0),-r_0,-r_1,-r_2,-r_3,r_0,\ldots).
$$ 
In order to calculate six-anti-periodic sequences $(r_k)_{k\in\NN_0}$ (rather than looking at fixed points of the 
six-stage recursion \eqref{recursion:zzz}, which requires solving a sixth order polynomial), 
we observe that six-anti-periodic sequences $(r_k)_{k\in\NN_0}$ are build of three-anti-periodic patterns 
of the two-stage-recursion \eqref{recursion:y}, which 
exist if and only if
$$
r_1^2-2=1\iff r_1=\pm \sqrt{3}.
$$
In particular, one can verify that for any $r_0\in\RR$, it holds that $r_6=-r_0$ and $r_7=-r_1=\mp\sqrt{3}$, 
which leads, with Corollary~\ref{persol}, to the six-anti-periodic pattern
$$
(r_0,\pm\sqrt{3},3-r_0,\pm\sqrt{3}(2-r_0),3-2r_0,\pm\sqrt{3}(1-r_0),-r_0,-r_1,-r_2,-r_3,-r_4,-r_5,r_0,\dots).
$$ 

\end{ex}

The above example of finding six-anti-periodic sequences $(r_k)_{k\in\NN_0}$ indicates a practical way to determine 
periodic solutions, where the periodicity $p$ can be factorised into sufficiently small numbers. 
This is described in the following remark. 

\begin{rem}
In order to study, for instance, six-periodic sequences $(r_k)_{k\in\NN_0}$, instead of directly  
analysing the six-stage recursion \eqref{recursion:zzz} (which leads to sixth order polynomials), 
we can alternatively consider three-periodic patterns of the two-step recursion \eqref{recursion:y}, 
where we require that the recursion parameter 
$s_2=r_1^2-2$ of \eqref{recursion:y} equals minus one, i.e. 
$
r_1^2-2=-1\iff r_1=\pm 1.
$
Observe that only for $r_1=1$ are these patterns genuinely six-periodic (and three-anti-periodic), while 
for $r_1=-1$, we recover three-periodic patterns.  

Similarly, we can also consider two-periodic solutions of the three-step recursion \eqref{recursion:z}, 
where we require that the parameter 
$s_3 = r_1(r_1^2-3)$ of \eqref{recursion:y} equals minus two, i.e. 
$$
r_1(r_1^2-3)=-2\iff (r_1-1)^2(r_1+2)=0 \iff r_1=1 \ \vee \ r_1=-2,
$$
and recall that $r_1=-2$ leads to a two-periodic pattern.
\end{rem}

As a consequence of Proposition~\ref{prop:PeriodicSolutions}, 
we obtain the following results on the dynamics of the recursion \eqref{recursion:r} for $|r_1|<2$.

\begin{prop}[Dynamics of the recursion \eqref{recursion:r} for $|r_1|<2$.]\label{prop:S2}\hfill\\
Let $|r_1|<2$. Then, the growth coefficient recursion \eqref{recursion:r}, i.e.
\begin{equation*}
r_{k+2} = r_1 r_{k+1} - r_k, \qquad |r_1|<2,r_0 \in \RR.
\end{equation*}
yields sequences $(r_k)_{k\in\NN_0}$ satisfying the following cases of periodic behaviour:
\begin{itemize}[topsep=5pt, leftmargin=7mm]
\item If $r_1=0$, then \eqref{recursion:r} yields for all $r_0\in\RR$ the \emph{four-periodic pattern} 
$$
r_{2k+1} = 0,     \quad 
r_{2k} = (-1)^{k} r_0, \qquad k\in \NN_0.
$$
{As an example for this case, we have} 
$$
(r_0,r_1,r_2,r_3,r_4,r_5,\ldots) = (2,0,-2,0,2,0,\ldots), 
$$
which arises from the quiddity sequences
(1,3,2,1,3,2,\dots) or (1,2,3,1,2,3,\dots) from the zig-zag triangulations of a hexagon with $180^{\degree}$ symmetry.

\medskip

\item If $0<r_1<2$, then the recursion \eqref{recursion:xx} has no fixed point. Moreover, by using the arguments of the proof of Lemma \ref{lem:x}, the recursion \eqref{recursion:xx} is strictly monotone decreasing for $x_k>0$, strictly monotone  increasing for $x_k<0$ and anti-contractive if and only if $x_k\in\bigl(0,\frac{2}{r_1}\bigr)$.

\medskip 

In the special case $r_1=1$, Proposition~\ref{prop:PeriodicSolutions} implies that all sequences $(r_k)_{k\in\NN_0}$ are three-anti-periodic and thus six-periodic, i.e. that for any $r_0\in\RR$:
$$
\qquad 
r_{0+3l} = (-1)^l r_0, \quad 
r_{1+3l} = (-1)^l (-1), \quad 
r_{2+3l} = (-1)^l (1-r_0), \quad \forall l\in\NN_0.
$$
{As an example for this case, we have} 
$$(r_0,r_1,r_2,r_3,r_4,r_5,r_6,r_7,\ldots) = (2,1,-1,-2,-1,1,2,1,\ldots),$$ 
which arises from the quiddity sequence 
(1,3,1,3,1,3,\dots) from the triangulations with $120^{\degree}$ symmetry of the hexagon by an inner triangle.\medskip

In general, the sequence $(r_k)_{k\in\NN_0}$ will undergo infinitely many sign changes. It is an open problem for which $r_1\in(0,2)\setminus\{1\}$ there exist periodic or non-periodic sequences, how frequently it may occur that $r_k=0$ for some $k\in\NN_0$, and what kind of sign-changing periodic pattern or even strange attractors might exist.

\medskip

\item If $r_1<0$ and $r_2\in \RR$, the modified definition $\tilde x_{k+1} := -\frac{r_{k+2}}{r_{k+1}}$ (as already used in Proposition~\ref{prop:S}) yields the one-stage recursion \eqref{recursion:x} with parameter $|r_1|$. 
Thus, the above results carry over accordingly. 

In the \emph{special case $r_1=-1$}, Proposition~\ref{prop:PeriodicSolutions} implies, for all 
$r_0\in\RR$, the existence of a three-periodic pattern of the form 
$$
r_{0+3l} = r_0, \quad 
r_{1+3l} = -1, \quad 
r_{2+3l} = 1-r_0, \quad \forall l\in\NN_0.
$$
\end{itemize}
\end{prop}

\appendix
\section{Proof of Proposition~\ref{prop:relationvaluesA}~(\ref{prop:relationvaluesA2})}
\label{App:A}

\begin{proof}
Let $\mathcal F$ be an $n$-periodic frieze ($n$ the minimal period). 
We have the recursive formula 
$s_{k+2}=s_1s_{k+1}-s_k$ from Proposition~\ref{prop:relationvaluesA} (\ref{prop:relationvaluesA1}). 
This holds for $k\ge 0$, 
with $s_0=2$.  

Now, we want to prove the following closed formula: For $k\ge 1$, we have 
$$
s_k=s_1^k + k\sum_{l=1}^{\lfloor{\frac{k}{2}}\rfloor} (-1)^{l} \frac{1}{k-l}{k-l \choose l}s_1^{k-2l}.
$$

To prove this, we use induction on $k$. 
The claim is true for $k=1,2$. So assume the claim 
holds for $k-1$ and for $k$. By (\ref{prop:relationvaluesA1}), we have $s_{k+1}=s_1s_k-s_{k-1}$. 
We use the induction hypothesis  
to replace $s_k$ and $s_{k-1}$, thus obtaining the followng expression for $s_{k+1}$ (note that in the third line, the summands involving $s_1^{k-1}$ are grouped into one 
term and the sums are corrected accordingly): 

$
\begin{array}{cl}  
    & s_1\left(s_1^k+ k\sum\limits_{l=1}^{\lfloor{\frac{k}{2}}\rfloor} (-1)^{l}\, \frac{1}{k-l}{k-l\choose l}s_1^{k-2l}\right)  \\[1em]
      &    - \left( s_1^{k-1} + (k-1)\sum\limits_{m=1}^{\lfloor{\frac{k-1}{2}}\rfloor} (-1)^{m}\, 
      \frac{1}{k-1-m}{k-1-m\choose m}s_1^{k-1-2m}\right) \\ [1em]
    = & s_1^{k+1} - (k+1) s_1^{k-1} + \sum\limits_{l=2}^{\lfloor{\frac{k}{2}}\rfloor} (-1)^{l}\,\frac{k}{k-l} {k-l\choose l}s_1^{k+1-2l} 
    - \sum\limits_{m=1}^{\lfloor{\frac{k-1}{2}}\rfloor} (-1)^{m}\,\frac{k-1}{k-1-m}{k-1-m \choose m} s_1^{k-1-2m} \\[1em]
\end{array}
$

Recall that we want this to be equal to 
$s_1^{k+1}+ (k+1)\sum_{l=1}^{\lfloor\frac{k+1}{2}\rfloor} (-1)^{l} \frac{1}{k+1-l}{k+1-l\choose l} s_1^{k+1-2l}$.  
We already have equality for the summands involving $s_1^{k+1}$ and of $s_1^{k-1}$. 
We thus consider the remaining terms 
$$
 \sum_{l=2}^{\lfloor{\frac{k}{2}}\rfloor} (-1)^{l}\,\frac{k}{k-l} {k-l\choose l}s_1^{k+1-2l} 
    - \sum_{m=1}^{\lfloor{\frac{k-1}{2}}\rfloor} (-1)^{m}\,\frac{k-1}{k-1-m}{k-1-m \choose m} s_1^{k-1-2m},
$$
which must be shown to be equal to 
$ \sum_{l=2}^{\lfloor\frac{k+1}{2}\rfloor} (-1)^{l} \frac{k+1}{k+1-l}{k+1-l\choose l} s_1^{k+1-2l}$. 

In the next step, we take the 
first expression and write it in a single sum. \\
We first do this in the case where $k$ is even. In this case, we 
have $\lfloor \frac{k}{2}\rfloor=\frac{k}{2}$ and $\lfloor\frac{k-1}{2}\rfloor=\frac{k}{2} -1$. 
Changing the running index in the second sum, we write 
$\sum_{m=1}^{\lfloor{\frac{k-1}{2}}\rfloor} (-1)^{m+1}\,\frac{k-1}{k-1-m}{k-1-m \choose m} 
s_1^{k-1-2m}$ as 
$\sum_{l=2}^{\frac{k}{2}} (-1)^{l}\,\frac{k-1}{k-l}{k-l \choose l-1}s_1^{k+1-2l}$ 
to get \\[0.5em]
$
\begin{array}{cl} 
&  \sum\limits_{l=2}^{\frac{k}{2}} (-1)^{l}\,\frac{k}{k-l} {k-l\choose l}s_1^{k+1-2l} 
    +\sum\limits_{l=2}^{\frac{k}{2}} (-1)^{l}\,\frac{k-1}{k-l}{k-l \choose l-1}s_1^{k+1-2l}\\
    & \\
= & \sum\limits_{l=2}^{\frac{k}{2}} (-1)^{l} s_1^{k+1-2l} 
\underbrace{\left(\frac{k}{k-l}{k-l\choose l} + \frac{k-1}{k-l}{k-l\choose l-1} \right)}_{=:A_l}.
\end{array}
$\\[0.5em]
We want this sum to be equal to 
$ \sum_{l=2}^{\frac{k}{2}} (-1)^{l} \frac{k+1}{k+1-l}{k+1-l\choose l} s_1^{k+1-2l}$, so let us compare the coefficients 
$A_l$ and 
$B_l:=\frac{k+1}{k+1-l}{k+1-l\choose l}=\frac{k+1}{k+1-l}\frac{(k+1-l)!}{l!(k-2l+1)!}
=\frac{(k-l)!}{l!(k-2l+1)!}(k+1)$. 
$$
\begin{array}{lcl}
A_l  
 & = & \frac{k(k-l)!}{(k-l)l! (k-2l)! } + \frac{(k-1)(k-l)!}{(k-l)(l-1)!(k-2l+1)!} \\ 
 & & \\
 & = & \frac{k}{k-l}\frac{(k-l)!(k-2l+1)}{l!(k-2l+1)!} + \frac{k-1}{k-l} \frac{l (k-l)!}{l!(k-2l+1)!} \\
 & & \\
 & = & \frac{(k-l)!}{l!(k-2l+1)!} \frac{k(k-2l+1) + (k-1)l}{k-l} \\
\end{array}
$$
Since $k(k-2l+1) + (k-1)l = k^2-kl+k-l = (k-l)(k+1)$, we have $k+1=\frac{k(k-2l+1) + (k-1)l}{k-l}$ 
and thus $A_l=B_l$. 

The proof for odd $k$ works similarly. 
\end{proof}

\section{Lemma~\ref{lem:x}}\label{appb}

\begin{lem}\label{lem:x}
Let $s_1\ge2$ and consider, for $k\ge 1$, the recursion
\begin{equation}\label{recursion:x}
x_{k+1} = s_1  - \frac  {1}{x_k}=: f(s_1,x_{k}),\qquad x_1 \in \RR,
\end{equation}
which is generalised by the extension rule \eqref{extension}, i.e. 
if $x_k=0$ for some $k\ge1$, then we set $x_{k+1}=-\infty$, 
$x_{k+2}=s_1$ and continue \eqref{recursion:x}.

Then, the extended recursion \eqref{recursion:x} exhibits the stable fixed point
\begin{align}
S &:= \frac{1}{2}\left(1 + \sqrt{s_1^2-4}\right) \quad\text{with}\quad1\le S < \infty \quad \text{for} \quad s_1 \in [2,\infty),\label{xinfty1}\\
\intertext{and the unstable fixed point}
U&:= \frac{1}{2}\left(1 - \sqrt{s_1^2-4}\right) \quad\text{with}\quad
 0< U \le 1 \quad\ \text{for} \quad  s_1 \in [2,\infty).\label{xinfty2}
\end{align}
In particular, the stable fixed-point $S$ attracts all initial values $x_1\in \RR\setminus\{U\}$ exponentially fast in the following sense:
If $x_1\in(U,S)$, then for every $K\ge1$ we have $U<x_1< x_K<x_{K+1}< S$ and
\begin{equation}\label{leftconv}
S-x_{K+k} < \Bigl(\frac{1}{S x_K}\Bigr)^{k} (S-x_K), \qquad \forall \ k,K\ge1, 
\end{equation}
where $Sx_1>SU=1$. 
If $x_1\in \RR\setminus[U,S]$, then there exists an index $K\ge 1$ such that $S<x_K$. Moreover, it holds that $S<x_{K+k+1}<x_{K+k}<x_K$ for all $k\ge1$ and
\begin{equation}\label{rightconv}
x_{K+k} -S< 
\Bigr(\frac{1}{S^2}\Bigr)^{k} (x_K-S), \qquad \forall \ k\ge1.
\end{equation}

If $s_1=2$, then $S=U=1$ is a single unstable fixed point, which is repelling to the left yet attracts all initial values $x_1\in \RR$ from above, i.e. $\lim_{k\to\infty} x_{k} \searrow  S=1$ 
(which is a consequence of $-\infty\le x_k< 0$ implying $x_{k+1}\ge s_1>S$). 
\end{lem}
\begin{rem}
The recursion \eqref{recursion:x} needs to be extended by \eqref{extension} to pass through zero and minus infinity at most once. 
More precisely, the set of initial values $x_1\neq 0$, for which $x_k=0$ holds after finitely many iterations of the recursion \eqref{recursion:x} 
is a countable set $\mathcal{N}$ of initial values within the interval $\bigl[\frac{1}{s_1},U\bigr)$, 
\begin{equation}\label{null}
\mathcal{N} := \left\{ y_k \in \Bigl[\frac{1}{s_1},U\Bigr)\ \vline \ y_k = (f^{-1})^{(k)}(0) \quad \text{for all } k\ge 1 \right\},
\end{equation}
for which it follows from the proof of Lemma \ref{lem:x} that $U>y_{k+1} > y_{k}$ for all $k\ge 1$.
Thus, if for some $k\in\NN$, we are looking for $y_k = x_1 =\frac{s_2}{s_1}$, we then require $s_2 = s_1 y_k \in [1,s_1 U)\subset[1,2)$, where the last inclusion follows from $s_1 U =  U^2 +1$ (by definition) and $U^2 +1\le2$. 
As a consequence, if $x_k=0$, then $x_{k+2}=s_1\notin \mathcal{N}$ since $s_1\ge2$ and thus 
$x_{k+l}\neq 0$ for all $l\ge 2$, which means the recursion \eqref{recursion:x} will not pass through zero again. 
\end{rem}

\begin{proof}
Possible fixed points $x$ of the recursion \eqref{recursion:x} are determined by $x = s_1  - \frac  {1}{x}$, which is equivalent to solving the quadratic equation
$$
F(s_1,x) := x^2 - s_1 x + 1 =0 \quad\Longrightarrow 
\quad x_{1} = S,  \ x_{2} = U, 
$$
where $S$ and $U$ are given as in \eqref{xinfty1} and \eqref{xinfty2}, and 
the stated properties are easily verified. 

Next, it is straightforward to check the monotonicity property of the 
recursion \eqref{recursion:x} by calculating
$$
x_{k+1}-x_{k} = s_1 - \frac{1}{x_k}-x_k = -\frac{F(s_1,x_k)}{x_k}.
$$
Thus, by observing that $F(s_1,x)<0 \Longleftrightarrow x\in(U,S)\subset (0,+\infty)$, and upon defining the intervals 
$$
\RN{4}:=[-\infty,0),\quad
\RN{3}:=[0,U), \quad 
\RN{2}:=(U,S), \quad
\RN{1}:=(S,+\infty),  
$$
it follows with the extension rule \eqref{extension} that 
\begin{align*}
x_{k+1} < x_k &\Leftrightarrow \left[F(s_1,x_k) >0 \, \wedge \, x_k >0\right]  \, \vee \, \left[x_k=0 \, \wedge \ x_{k+1}=-\infty \right]\Leftrightarrow x_k \in \RN{3}\cup\RN{1},\\
x_{k+1} > x_k &\Leftrightarrow 
\left[F(s_1,x_k) <0 \, \wedge \, x_k >0 \right] \, \vee \,
\left[F(s_1,x_k) >0 \, \wedge \, x_k <0 \right]  \Leftrightarrow x_k \in \RN{4}\cup\RN{2}.
\end{align*}

Moreover, we have $\frac{\partial }{\partial x}f(s_1,x) = x^{-2}>0$, which means 
that the recursion \eqref{recursion:x} is order preserving, i.e. 
$$
x_k  \le \bar{x}_k \quad\Longrightarrow\quad x_{k+1} \le \bar{x}_{k+1}.
$$ 

Finally, we observe that $x_{k+1}=0 \iff x_{k} = \frac{1}{s_1}$, and that 
$\frac{1}{s_1} < U$ for all $s_1\ge2$, which motivates us to define
$$
\RN{3} = \RN{3}_b  \cup \RN{3}_a := 
\Bigl[0,\frac{1}{s_1}\Bigr) \cup \Bigl[\frac{1}{s_1},U\Bigr).
$$

Altogether, we conclude that the mapping $f$ satisfies the following monotonicity properties:  

\begin{align}\label{monotone}
f : 
\begin{cases}
\RN{1} \mapsto (S, s_1)\subset\RN{1} &\text{ bijective and order preserving with }
x_{k+1}< x_k,\\
\{S\} \mapsto \{S\} &\text{ with } x_{k+1}=f(S)= S=x_k,\\
\RN{2} \mapsto \RN{2} &\text{ bijective and order preserving with }
x_{k+1} > x_k,\\
\{U\} \mapsto \{U\} &\text{ with } x_{k+1}=f(U)= U=x_k, \\
\RN{3}_a \mapsto \RN{3} &\text{ bijective and order preserving with }
x_{k+1}< x_k,\\
\RN{3}_b \mapsto \RN{4} &\text{ bijective and order preserving with }
x_{k+1}< x_k,\\
\RN{4} \mapsto [s_1,+\infty)\subset\RN{1}& \text{ bijective and order preserving with }
x_{k+1} > x_k.\\
\end{cases}
\end{align}
Therefore, according to the monotonicity properties \eqref{monotone},
all initial values $x_1\in \RN{4}\cup\RN{2}\cup\RN{1}$ as well as
all initial values $x_1\in\RN{3}\,\setminus\,\mathcal{N}$ will converge to $S$ without passing through zero. Here the set $\mathcal{N}$ as defined in \eqref{null} is the set of all values, which are mapped onto $0$ during the recursion and for which the recursion has to be extended by rule \eqref{extension} to pass through zero and minus infinity. The set $\mathcal{N}$ is obtained from considering the backward recursion 
$$
x_{k} = f^{-1}(x_{k+1})=\frac{1}{s_1-x_{k+1}}, 
$$
and the same monotonicity arguments as above imply that the backward recursion
$f^{-1}$ restricted to  $[0,U) \mapsto  \bigl[\frac{1}{s_1},U\bigr)$
is strictly monotone increasing, which shows $\mathcal{N}\subset \bigl[\frac{1}{s_1},U\bigr)$ and $(f^{-1})^{k+1}(0)>(f^{-1})^{k}(0)$, verifying \eqref{null}.
\medskip

Finally, the stated rates of convergence follows from observing that
\begin{equation}\label{contr}
x_{k+1}-S = r_k (x_k-S), \qquad\text{where}\quad r_k = \frac{1}{S x_k}.
\end{equation}
Since for $x_1\in(U,S)$, the sequence $(x_k)_{k\in\NN}$ is strictly monotone increasing and $\lim_{k\to\infty}x_k\nearrow S$, the statement \eqref{leftconv} follows directly from \eqref{contr}, which is a contraction since $(r_k)^{-1}\ge (r_1)^{-1}=Sx_1>SU=1$ for $k\ge 1$. 
Similarly, for $0\le x_1<U$
the recursion \eqref{recursion:x} is monotone decreasing and there exists an index $K\ge3$ such that $x_{K-1}<0$, and thus $x_K\ge s_1>S$ (this includes the case $x_{K-2}=0$, $x_{K-1}=-\infty$ and $x_K=s_1$ arising via the extension rule \eqref{extension}). 
Thus, if $x_1\in\RR\setminus[U,S]$, there exists an index $K\ge1$
such that $S<x_K$. We then have that 
$S<x_{K+k}$ for all $k\ge0$ and $\lim_{k\to\infty} x_{K+k}\searrow S$.
Moreover, since $1>S^{-2}>r_{K+k}\ge r_{K}$ for all $k\ge0$, the contraction \eqref{contr} directly yields the exponential convergence \eqref{rightconv}.
\end{proof}

\subsection*{Acknowledgements}
The authors thank Gregg Musiker and Hannah Vogel for pointing out the connection between Proposition~\ref{prop:relationvaluesA} and Chebyshev polynomials. 
The first and second authors were supported by the Project ``Mathematics and Arts" granted by the University of Graz. 
The first and third authors gratefully acknowledge support by the Austrian Science Fund (FWF): Project No.\ P25141-N26. 
The first author thanks support through FWF project W1230. 
In addition, all authors acknowledge support from NAWI Graz.

\end{document}